\documentclass[journal]{IEEEtran}
\usepackage{balance}
\usepackage{color}
\usepackage{subfigure}
\usepackage{cite}
\usepackage{enumerate}
\usepackage{empheq}


\usepackage{graphicx,amssymb,amstext,amsmath,amsthm}

\usepackage[bookmarks=true]{hyperref}
\usepackage{algorithm,algorithmicx,algpseudocode}
\algnewcommand{\algorithmicgoto}{\textbf{go to}}%
\algnewcommand{\Goto}[1]{\algorithmicgoto~\ref{#1}}%
\algnewcommand{\LineComment}[1]{\Statex \(\triangleright\) #1}

\usepackage{multirow}
\usepackage{stfloats}

\newtheorem{innercustomthm}{Assumption}
\newenvironment{assump}[1]
  {\renewcommand\theinnercustomthm{#1}\innercustomthm}
  {\endinnercustomthm}

\newtheorem{prop}{Proposition} 
\newtheorem{cor}{Corollary}
\newtheorem{thm}{Theorem}
\newtheorem{lem}{Lemma}
\newtheorem{defn}{Definition}
\newtheorem{rem}{Remark}

\newtheorem{spc}{Special Case}

\providecommand{\doverline}[1]{\overline{\overline{#1}}}

\begin{document}

\title{Simultaneous Input and State Estimation for Linear Time-Varying Continuous-Time Stochastic Systems$^\star$}

\author{%
Sze Zheng Yong$\,^{a,\dagger}$ \quad\quad Minghui Zhu$\,^b$ \quad\quad Emilio Frazzoli$\,^{c,\dagger}$\\
\thanks{
$^a$ S.Z Yong is with the School
for Engineering of Matter, Transport and Energy, Arizona State University, Tempe, AZ 85281, USA (e-mail: szezheng.yong@asu.edu).}
\thanks{
$^b$ M. Zhu is with the Department of Electrical Engineering, Pennsylvania State University, University Park, PA 16802, USA (e-mail: muz16@psu.edu).}
\thanks{
$^c$ E. Frazzoli is with the Institute for Dynamic Systems and Control, Swiss Federal Institute of Technology (ETH), CH-8092 Z\"{u}rich, Switzerland (e-mail: efrazzoli@ethz.ch).}
\thanks{
$^\dagger$ This work was done when S.Z. Yong and E. Frazzoli were with the
Laboratory for Information and Decision Systems at Massachusetts Institute
of Technology, Cambridge, MA 02139, USA.
}
\thanks{
$^\star$ Extended version of an IEEE Transactions on Automatic Control paper with the same title.
}
}

\maketitle

\begin{abstract}
In this paper, we present an optimal filter for linear time-varying continuous-time stochastic systems that simultaneously estimates the states and unknown inputs in an unbiased minimum-variance sense. We first show that the unknown inputs cannot be estimated without additional assumptions. Then, we  discuss two complementary variants of the filter: (i) for the case when an additional measurement containing information about the state derivative 
is available, and (ii) for the case without the additional measurement 
but the input signals are assumed to be sufficiently smooth and have bounded derivatives. 
Conditions for uniform asymptotic stability and the existence of a steady-state solution for the proposed filter, as well as the convergence rate of the state and input estimate biases are given. Moreover, we show that a principle of separation of estimation and control holds and that the unknown inputs may be rejected. Two examples, including a nonlinear vehicle reentry example, are given to illustrate that our filter is applicable even when some strong assumptions do not hold. 
\end{abstract}

\section{Introduction}

When the inputs to linear continuous-time stochastic systems are known, the Kalman-Bucy filter \cite{kalman.bucy.1961} provides the optimal state filtering solution from noisy measurements. However, in many applications, the disturbance inputs or the unknown parameters are not modeled by a zero-mean, Gaussian white noise. For instance, (semi-)autonomous vehicles do not have knowledge of the control inputs of other vehicles. 
The inability to reliably track the states of these vehicles, or to estimate the unknown inputs may lead to a collision or suboptimal performance, etc.
Similar problems are found across many disciplines, e.g., meteorology \cite{Kitanidis.1987}, physiology \cite{DeNicolao.1997}, fault detection and diagnosis \cite{Patton.1989} and machine tool applications \cite{Corless.98}; hence, a solution to this problem is beneficial for a wide range of applications.

\emph{Literature review.} 
Research in this field began with state estimation of systems with unknown biases 
and unknown disturbance of known dynamics, 
but has since moved towards state estimation with arbitrary unknown inputs. An optimal filter that only estimates the system states in a minimum-variance unbiased (MVU) sense estimate is first developed for linear discrete-time stochastic systems with unknown inputs in \cite{Kitanidis.1987,Darouach.1997,Hou.1998,Darouach.2003,Cheng.2009}. This development was followed by the design of optimal simultaneous input and state estimation filters, with the objective of \emph{concurrently} obtaining minimum-variance unbiased estimates for both the states and the unknown disturbance inputs to the system, as researchers realize that the information about the unknown input is often as important as state information. However, initial research has been focused on particular classes of linear discrete-time systems with unknown inputs 
(see e.g., \cite{Gillijns.2007b,Fang.2008,Fang.2011,Yong.Zhu.Frazzoli.2013} and references therein). 
Only recently has a general framework been proposed in \cite{Yong.Zhu.ea.Automatica16,Yong.Zhu.ea.CDC15_delay} for optimally estimating both state and unknown input of linear discrete-time stochastic systems with unknown inputs. 

To our best knowledge, the problem of simultaneous state and input estimation for linear continuous-time stochastic systems has not been addressed. Thus, we turn to the literature on unknown input observer designs for deterministic systems for inspiration.
As it turns out, the accessibility of output derivatives plays an important role for the estimation of the unknown inputs  in observer designs.
Some observer designs (e.g., \cite{Hou.1998b}) 
differentiates the output measurements, 
whereas other designs (e.g., \cite{Corless.98,Xiong.2003}) 
rely solely on output measurements without differentiation, although these observers can only asymptotically estimate the unknown input to any degree of accuracy instead of exact asymptotic estimation.

\emph{Contributions.} 
We propose a \emph{stable} and \emph{optimal} state and unknown input filter in the minimum-variance unbiased sense for linear \emph{time-varying} continuous-time stochastic systems and provide the convergence rate of the proposed filter. First, we show via a similarity transformation that the unknown input is in general not directly observable from the output signal and hence, unlike its discrete-time counterpart, cannot be estimated in a meaningful way without additional assumptions. Then, taking a leaf out of deterministic observer designs (e.g., \cite{Corless.98,Hou.1998b}), we provide an analysis of two sets of assumptions under which the input can be estimated: (i) when an additional measurement containing information about the state derivative or `output derivative' is available, and (ii) when no additional measurement is accessible but the input signals are sufficiently smooth and have bounded derivatives. 

Two complementary variants of the optimal filter are presented for each of these assumptions. In the latter case, as with observer designs in \cite{Corless.98,Xiong.2003}, where exact asymptotic estimation is not available, 
we propose a filter variant that still estimates the system states in an MVU sense, but the unknown inputs are only estimated to any degree of accuracy when compared to the MVU input estimate obtained if the exact output derivative is known. The proposed filter is derived by constructing a `virtual' equivalent system without unknown inputs\footnote{The proof technique of constructing a virtual system, while is rather common for controller designs, is to our knowledge novel to filter designs.}. Although not implementable, this ÔvirtualÕ system without unknown inputs has provably the same properties as our proposed filter, allowing us to derive analogous properties of our filter to that of the well-known Kalman-Bucy filter  \cite{kalman.bucy.1961}. Moreover, by limiting case approximations of the optimal discrete-time filter in \cite{Yong.Zhu.ea.Automatica16}, we find that the discrete-time filter implicitly uses finite difference to obtain an `output derivative'.

Moreover,  
the derivatives of the system matrices may be needed, where the main challenge lies in the computation of derivatives of the singular value decomposed matrices of the direct feedthrough matrix. A solution to this problem is presented in Section \ref{sec:SVD}, which, as a by-product, provides a novel alternative approach to \cite{Wright.1992,Bunse.91} for computing analytic singular value decomposition with differential equations. 

Finally, we show that a principle of separation of estimation and control also exists for linear systems with unknown inputs, and that the unknown inputs may be rejected, if desired. Hence, we can combine the proposed stable filter for state and input estimation, with any independently designed stable state feedback controller to achieve a stable closed loop system, which we illustrate with a vehicle reentry example with nonlinear dynamics \cite{Julier.2004} and a helicopter hover control example in windy environments even when some strong assumptions in our paper do not hold. A preliminary version of this paper is presented in \cite{Yong.Zhu.Frazzoli.ACC15} where the special case of linear time-invariant systems is studied. 

\emph{Notation.} We first summarize the notation used in the paper. $\mathbb{R}^n$ denotes the $n$-dimensional Euclidean space.
For a vector $v \in \mathbb{R}^n$, its $r^{th}$ derivative is denoted by $v^{(r)}$ and its expectation by $\mathbb{E}[v]$. Given a matrix $M \in \mathbb{R}^{p \times q}$, its transpose, inverse, Moore-Penrose pseudoinverse, norm, trace, rank are given by $M^\top$, $M^{-1}$, $M^{\dagger}$,  $\|M \|$, ${\rm tr}(M)$ and ${\rm rk}(M)$.
For a symmetric matrix $S$, $S \succ 0$ ($S \succeq 0$) is positive (semi-)definite. 

\section{Problem Statement} \label{sec:Problem}

We consider the following model representation of linear time-varying continuous-time stochastic systems
\begin{align} \label{eq:system}
\hspace{-0.1cm} \begin{array}{ll}
\dot{x} (t) &=A(t) x (t) +B(t) u(t) +G(t) d(t) + W(t) w(t),\\
y(t) &= C(t) x(t) +D(t) u(t) + H(t) d(t) + v(t), \end{array} 
\end{align} \noindent where $x(t) \in \mathbb{R}^n$ is the state vector at time $t$, $u (t) \in \mathbb{R}^m$ a known input vector, $d (t) \in \mathbb{R}^p$ an unknown input vector, $y(t) \in \mathbb{R}^l$ the measurement vector, $w(t) \in \mathbb{R}^q$ the process noise and $v(t) \in \mathbb{R}^l$ the measurement noise. 
The matrices $A(t)$, $B(t)$, $G(t)$, $C(t)$, $D(t)$, and $W(t)$
are smooth, bounded and known, whereas $H(t)$ is analytic (i.e., infinitely differentiable and convergent) and known. $x(t_0)=x_0$ is also assumed to be independent of $v(t)$ and $w(t)$ for all $t$ and an initial state estimate $\hat{x}(t_0):=\hat{x}_0$ is available with covariance matrix $\mathcal{P}_{0}^x$. Without loss of generality, we assume that $n \geq l \geq 1$, $l \geq p \geq 0$ and $m \geq 0$ and the current time $t$ is strictly positive. Our fairly general time-varying system formulation facilitates linearization-based nonlinear filtering techniques, as is demonstrated in our simulation example in Section \ref{sec:reentry}. To simplify notations, we often omit the explicit time-dependence of signals when it is clear from context.

It has been observed in \cite{Hou.1998b} that, except for some trivial cases (e.g., $H$ has full rank), derivatives of outputs are needed when the reconstruction of the unknown input is desired for deterministic systems. Therefore, we expect stochastic systems to similarly require some form of additional signal information that is a counterpart of the output derivative in the deterministic case. With this in mind, we first show via a similarity transformation in Proposition \ref{prop1} that the unknown input is indeed not directly observable from the output signal and thus, unlike its discrete-time counterpart, cannot be estimated in a meaningful way without additional assumptions.  

\emph{Objective.} The objective of this paper is hence to design an optimal recursive filter algorithm which simultaneously estimates the system state $x(t)$ and the unknown input $d(t)$ based on an initial state estimate $\hat{x}_0$ with covariance $\mathcal{P}^x_0$, and measured outputs up to time $t$, $y(\tau)$ for all $0 \leq \tau \leq t$, under some appropriate assumptions (to be explored in Section \ref{sec:MVUE}). No prior knowledge of the dynamics of $d(t)$ is assumed.

\section{Preliminary Material} \label{sec:prelim}

We begin by providing the definition of uniform complete controllability and observability:

\begin{defn}[Uniform Complete Controllability \& Observability\cite{kalman.bucy.1961,Kalman.1960}] \label{def:uco}
Let $X(t)$ be bounded. The pair $(X(t),Y(t))$ is uniformly completely controllable, if $\exists \epsilon>0$ and $\mu_1(\epsilon)>0$, $\mu_2(\epsilon)>0$, such that for all $t \geq t_0$, such that
$\mu_1(\epsilon) \leq \int^t_{t-\epsilon} \Phi_{X(t)}(t,s) Y(s) Y(s)^\top \Phi_{X(t)}^\top(t,s) ds \leq \mu_2(\epsilon)$,
where $\Phi_{X(t)}(t,s)$ is the transition matrix of the system $\dot{x}(t)=X(t) x(t)+Y(t) u(t)$ and $y(t)=Z(t) x(t)$. Similarly, the pair $(X(t),Z(t))$ is uniformly completely observable, if its dual pair $(X^\top(t),Z^\top(t))$ is uniformly completely controllable.
\end{defn}

In the following, we present the similarity transformation that decouples the output signal with respect to the unknown inputs, revealing that a certain component of the unknown inputs cannot be observed from the output signal. Then, we introduce a novel alternative approach to \cite{Wright.1992,Bunse.91} to obtain the derivative of singular value decomposed matrices of time-varying $H(t)$ that is needed for the development of our filter.

\subsection{Decoupling via Similarity Transformation} \label{sec:similar}
Similar to its discrete-time counterpart \cite{Yong.Zhu.ea.Automatica16}, we first carry out a transformation of the system. Let ${\rm rk} (H)=p_{H}$. Then, we rewrite 
$H$ using singular value decomposition (SVD) as
\begin{align} \label{eq:H}
H=\begin{bmatrix}U_{1}& U_{2} \end{bmatrix} \begin{bmatrix} \Sigma & 0 \\ 0 & 0 \end{bmatrix} \begin{bmatrix} V_{1}^{\, \top} \\ V_{2}^{\, \top} \end{bmatrix} = U_1 \Sigma V_1^\top=:H_1 V_1^\top,
\end{align} \noindent
where $\Sigma \in \mathbb{R}^{p_{H} \times p_{H}}$ is a diagonal matrix of full rank, with $U_{1} \in \mathbb{R}^{l \times p_{H}}$, $U_{2} \in \mathbb{R}^{l \times (l-p_{H})}$, $V_{1} \in \mathbb{R}^{p \times p_{H}}$, $V_{2} \in \mathbb{R}^{p \times (p-p_{H})}$ and $0$ matrices of appropriate dimensions. $U:=\begin{bmatrix} U_{1} & U_{2} \end{bmatrix}$ and $V:=\begin{bmatrix} V_{1} & V_{2} \end{bmatrix}$ are unitary matrices. Note that when $H$ is the zero matrix, $\Sigma$, $U_{1}$ and $V_{1}$ are empty matrices, and $U_{2}$ and $V_{2}$ are arbitrary unitary matrices. 
Then, we define two orthogonal components of the unknown input given by
$d_{1}:=V_{1}^\top d$ and 
$d_{2}:=V_{2}^\top d$.
Since $V$ is unitary, $d =V_{1} d_{1}+V_{2} d_{2}$. Next, we decouple the output $y$ using a nonsingular transformation
\begin{align} \label{eq:T_k}
T =\begin{bmatrix} T_{1} \\ T_{2} \end{bmatrix}
= \begin{bmatrix} I_{p_{H}} & -U_{1}^\top R U_{2} (U_{2}^\top R U_{2})^{-1}\\ 0 & I_{(l-p_{H}) } \end{bmatrix} \begin{bmatrix} U_{1}^\top \\ U_{2}^\top \end{bmatrix},
\end{align} \noindent
to obtain
\begin{align}
\begin{array}{l}
\dot{x}  =A x +B u+G_{1} d_{1} +G_{2} d_{2} + W w, \\ 
z_{1}=T_{1} y = C_{1} x + D_{1} u +\Sigma d_{1} + v_{1},\\ z_{2}=T_{2} y = C_{2} x + D_{2} u + v_{2},\\
\end{array}\hspace{-0.5cm}   \label{eq:sys}
\end{align} \noindent
where 
$C_{1} :=T_{1} C$, $C_{2} := T_{2} C = U_{2}^\top C$, 
$D_{1} :=T_{1} D$, $D_{2} := T_{2} D = U_{2}^\top D$, 
$G_{1} :=G V_{1}$, $G_{2} :=G V_{2}$, 
$v_{1} :=T_{1} v$ and $v_{2} := T_{2} v = U_{2}^\top v$. 
The transform was also chosen such that the measurement noise terms for the decoupled outputs are uncorrelated with each other, the process noise and the initial state, with the non-zero autocorrelations of $v_{1}$ and $v_{2}$ given by $R_1:=T_{1} R T_{1}^\top \succ 0$ and $R_2:=T_{2} R T_{2}^\top \succ 0$, respectively. 

With the above decoupling of the output signals with respect to the unknown inputs, we obtain the following proposition:

\begin{prop} \label{prop1}
The output $y$ contains insufficient information to fully estimate the signal $d$, specifically the component $d_2$, which does not appear in $z_1$ and $z_2$ (and $y=T^{-1} [z_1^\top \ z_2^\top]^\top$).
\end{prop}

\subsection{Computation of derivative of singular value decomposed matrices of time-varying ${H}(t)$} \label{sec:SVD}

To perform the decoupling transformation, 
the computation of the derivative of singular value decomposed matrices of the time-varying $H(t)$ may be needed, which we now derive. 
With the assumption that the matrix $H$ is analytic, \cite{Bunse.91} established the existence of a singular value decomposition of $H$ where the factors are also analytic functions, which they termed analytic singular value decomposition (ASVD).  This has the implication that $U_1$, $U_2$, $V_1$, $V_2$ and $\Sigma$ are differentiable. Next, we provide an approach motivated by \cite{Wright.1992,Bunse.91} 
for obtaining the signal derivatives, $\dot{U}_1$ and $\dot{\Sigma}$,
which are required to compute $\dot{H}_1:=\dot{U}_1 \Sigma + U_1 \dot{\Sigma}$,
as well as $\dot{U}_2$, $\dot{\Sigma}$, $\dot{V}_1$ and $\dot{V}_2$. 
For simplicity, we shall first assume
that the rank of matrix $H$ is constant, and that all singular values
remain positive. The generalization to the case when the singular values can become zero will be discussed in Remark \ref{rem4}. 

\begin{thm} \label{thm:Hdot}
Let $\Sigma=diag(\sigma_1, \sigma_2, \hdots, \sigma_{p_H})$ be such that $\sigma_i >0$ for all $ i =1,2, \hdots, p_H$. Then, the singular value decomposed matrices of the known derivative of $H$ in \eqref{eq:H}  given by
$\dot{H}= \dot{U}_1 \Sigma V_1^\top + U_1 \dot{\Sigma} V_1^\top + U_1 \Sigma \dot{V}_1^\top$,
can be found using
\begin{align}
\begin{array}{rl}
\dot{\Sigma}&=diag(\dot{\sigma}_1, \dot{\sigma}_2, \hdots, \dot{\sigma}_{p_H}), \ 
\dot{U}_1= U_1 E, \quad \dot{U}_2=0, \\ 
\dot{V}_1 &=V_1 F, \quad \ \dot{V}_2=0,\end{array}\label{eq:Vdot}
\end{align} \noindent
with initial conditions determined by $H(t_0)=U_1(t_0) \Sigma(t_0) V_1^\top(t_0)$ and 
whose components for all $i =1,2, \hdots, p_H$ can be computed as follows:
\begin{align} \label{eq:sigmadot}
&\dot{\sigma}_i=(U_1^\top \dot{H} V_1)_{ii},\\
&\begin{array}{ll}E_{ii} =F_{ii}=0, \ E_{ij}=\frac{\sigma_j (U_1^\top \dot{H} V_1)_{ij}+ \sigma_i (U_1^\top \dot{H} V_1)_{ji}}{\sigma_j^2-\sigma_i^2}, \\ 
F_{ij}=\frac{\sigma_j (U_1^\top \dot{H} V_1)_{ji}+ \sigma_i (U_1^\top \dot{H} V_1)_{ij}}{\sigma_j^2-\sigma_i^2},\end{array} \label{eq:EF}
\end{align} \noindent
if $\sigma^2_i \neq \sigma^2_j$. In the case that $\sigma^2_i = \sigma^2_j$,
if we have ${\sigma}_i^{(r)} \neq {\sigma}_j^{(r)}$ for some $r$, then
the solution for $E_{ij}$ and $F_{ij}$ is unique and can be found by
differentiating \eqref{eq:H} $r$ times\footnote{The explicit equations for each case are lengthy and
interested readers are referred to \cite{Wright.1992}.}.
If all derivatives are equal, e.g., when $H$ is a constant matrix, then, with $\sigma_{i} \neq 0$,
\begin{align}
\begin{array}{rl}
E_{ij}&=-F_{ij}=\frac{(U_1^\top \dot{H} V_1)_{ij}}{2 \sigma_{i}}, \ \ \mathrm { if }\ \sigma_i=\sigma_j, \label{eq:E1}\\ 
E_{ij}&=F_{ij}=-\frac{(U_1^\top \dot{H} V_1)_{ij}}{2 \sigma_{i}}, \ \ \mathrm { if }\ \sigma_i=-\sigma_j.\end{array} 
\end{align}
\end{thm}
\begin{proof}
Differentiating both sides of $H=U_1 \Sigma V_1^\top$, we have
\begin{align*}
&\dot{H}= \dot{U}_1 \Sigma V_1^\top + U_1 \dot{\Sigma} V_1^\top + U_1 \Sigma \dot{V}_1^\top \\
&\Rightarrow U_1^\top \dot{H} V_1=U_1^\top \dot{U}_1 \Sigma + \dot{\Sigma} + \Sigma \dot{V}_1^\top V_1.
\end{align*}
Next, as is done in \cite{Wright.1992},  
we define the matrices $E:=U_1^\top \dot{U}_1$ and $F:=V_1^\top \dot{V}_1$ which are both skew symmetric, as can be shown by differentiating $U_1^\top U_1=I$ and $V_1^\top V_1=I$ on both sides. Hence, we can find the derivative of $\Sigma$ with
\begin{align} \label{eq:Hdot}
\dot{\Sigma}&=U_1^\top \dot{H} V_1-E \Sigma + \Sigma F.
\end{align} \noindent
To obtain $\dot{U}_1$ from $E:=U_1^\top \dot{U}_1$, we first note that the linear system is in general, underdetermined (except when $H$ has full rank). Hence, $\dot{U}_1$  is not unique. We choose the minimum Frobenius norm solution given by $\dot{U}_1=(U_1^\top)^\dagger E$, which is equivalent to 
\eqref{eq:Vdot} because $U_1$ is orthonormal. 
It remains an open question as to whether there exists a better choice of $\dot{U}_1$, but we do not expect changes in this respect.

To obtain $\dot{U}_2$, we differentiate $U_1^\top U_2=0$ (obtained from the orthogonality of columns of $U$): 
\begin{align*}
\dot{U}_1^\top U_2 + U_1^\top \dot{U}_2=E^\top U_1^\top U_2+U_1^\top \dot{U}_2=U_1^\top \dot{U}_2=0. 
\end{align*}
Similar to the case for $\dot{U}_1$, the above linear system is underdetermined and thus, $\dot{U}_2$ is not unique. Once again, we choose $\dot{U}_2$ as in \eqref{eq:Vdot} such that its Frobenius-norm is minimized\footnote{
Note that this choice of $\dot{U}=\begin{bmatrix} \dot{U}_1 & \dot{U}_2 \end{bmatrix}$ is
equivalent to the minimization of total variation (or arc length)
in \cite{Bunse.91}.}. Likewise, $\dot{V}_1$ and $\dot{V}_2$ can be similarly obtained  and are given in \eqref{eq:Vdot}.

The diagonal terms of the skew-symmetric matrices $E_{ii}$ and $F_{ii}$, for all $i=1,2,\hdots,p_H$  are zero. Hence, the diagonal entries of $E \Sigma$ and $\Sigma F$ are also zero.
Since $\Sigma=diag(\sigma_1, \sigma_2, \hdots, \sigma_{p_H})$ is diagonal, its singular values can be computed from \eqref{eq:Hdot} as given in \eqref{eq:sigmadot}.
The off-diagonal terms can be computed from the algebraic constraints of \eqref{eq:Hdot} ($i \neq j$):
\begin{align}\label{eq:constraint}
\begin{array}{l}
0=(U_1^\top \dot{H} V_1)_{ij}-E_{ij}\sigma_j+\sigma_i F_{ij}, \\
0=(U_1^\top \dot{H} V_1)_{ji}-E_{ji}\sigma_i+\sigma_j F_{ji},\end{array}
\end{align} \noindent
and if $\sigma_i^2 \neq \sigma_j^2$, using the skewness of $E$ and $F$, we get the expressions in \eqref{eq:EF}. 
If $\sigma_i^2 = \sigma_j^2$, then a similar but longer argument shows that the differentiation of \eqref{eq:Hdot} will provide equations for determining $E_{ij}$ and $F_{ij}$ if $\dot{\sigma}_i \neq \dot{\sigma}_j$. This process may be repeated if equality holds and the solution is unique if  ${\sigma}_i^{(r)} \neq {\sigma}_j^{(r)}$ for some $r$. The expressions for these cases are lengthy and the readers are referred to \cite{Wright.1992} for the extended derivation and discussion. If all derivatives are equal, e.g., when $H$ is a constant matrix, this corresponds to the non-uniqueness of the solution to the singular value decomposition \cite{Wright.1992,Bunse.91}, in which case we can choose $E_{ij}$ and $F_{ij}$ such that the Frobenius norms of $E$ and $F$, and hence of $\dot{U}_1$ and $\dot{U}_2$, are minimized. This can be solved by minimizing $E_{ij}^2+F_{ij}^2$ subject to equality constraint \eqref{eq:constraint} with either $\sigma_i=\sigma_j$ or $\sigma_i=-\sigma_j$, for which the explicit solutions are given in \eqref{eq:E1}. 
\end{proof}

A useful corollary to the above theorem is as follows:
\begin{cor} \label{cor:cor1}
$T_2 H_1 =0$, $T_2 \dot{H}_1=0$ and $T_2 \ddot{H}_1=0$, where $H_{1} :=H V_{1}=U_{1} \Sigma$.
\end{cor}
\begin{proof}
Using $\dot{U}_1$ from \eqref{eq:Vdot} and the definition of $T_2$ in \eqref{eq:T_k}, we find
$T_2 H_1 =  T_2 U_1 \Sigma =0$, 
$T_2 \dot{H}_1=T_2 U_1 E \Sigma + T_2 U_1 \dot{\Sigma}=0$, 
$T_2 \ddot{H}_1=T_2 U_1 E^2 \Sigma +T_2 U_1 \dot{E} \Sigma + 2 T_2 U_1 E  \dot{\Sigma}+ T_2 U_1 \ddot{\Sigma}=0$. 
\end{proof}

\begin{rem} \label{rem4}
To compute ASVD for the general case when some singular values become zero,
we first note that the factors of $H= \tilde{U} \tilde{\Sigma} \tilde{V}^\top$
do not have a nice structure as in \eqref{eq:H}, i.e.,
there is no guarantee that for
$\tilde{\Sigma}=\begin{bmatrix} diag(\tilde{\sigma}_1,
\tilde{\sigma}_2, \hdots, \tilde{\sigma}_{p})  \\ 0_{(l-p) \times p} \end{bmatrix}$,
only the first $p_{H}$ diagonal entries
$\tilde{\sigma}_1, \tilde{\sigma}_2, \hdots, \tilde{\sigma}_{p_H}$ are non-zeros and the rest $\tilde{\sigma}_{p_H+1}, \hdots, \tilde{\sigma}_{p}$ are zeros. Furthermore, $\tilde{\sigma}_i=0$, for any $i=1,\hdots,p$ does not imply that $\dot{\tilde{\sigma}}=0$, as the rank of $H$ given by $p_H$ may increase. Without going into the details as this is more of an implementation issue, we would like to note that slight modifications to Theorem \ref{thm:Hdot} can be carried out to account for the case when any singular value becomes zero. This relies on careful accounting of the cases when $\tilde{\sigma}_i$ is zero or not, and partitions $\tilde{U}$ and $\tilde{V}$ into $\tilde{U}_1$ and $\tilde{U}_2$, as well as $\tilde{V}_1$ and $\tilde{V}_2$, respectively, where $\tilde{U}_1$ and $\tilde{V}_1$ are concatenations of all columns of $\tilde{U}$ and $\tilde{V}$, for which $\tilde{\sigma}_i \neq 0$, for all $i=1, \hdots, p$, whereas $\tilde{U}_2$ and $\tilde{V}_2$ are concatenations of the rest of the columns of $\tilde{U}$ and $\tilde{V}$.
The other necessary modification is the replacement of \eqref{eq:Vdot} with $\dot{\tilde{\sigma}}_i= (\tilde{U}^\top \dot{H} \tilde{V})_{ii}$ for all $i=1,\hdots,p$.
\end{rem}

\section{Algorithms for Minimum-variance Unbiased Estimation of State and Input} \label{sec:MVUE}

Since we have shown in Proposition \ref{prop1} that the unknown input is in general not directly observable from the output signal unless an `output derivative' signal is available, we now analyze two sets of assumptions under which the input can be estimated that are inspired by deterministic observer designs (e.g., \cite{Corless.98,Hou.1998b}), and propose two corresponding optimal estimator designs:
\begin{enumerate}[A.]
\item \emph{Exact Linear Input \& State Estimator} (ELISE), in which we assume that an additional `output derivative' measurement is available (inspired by the output differentiation approach in \cite{Hou.1998b}); 
\item \emph{Approximate Linear Input \& State Estimator} (ALISE), in which output derivative is not measured but the input signals are sufficiently smooth and have bounded derivatives (inspired by the derivative free approach in \cite{Corless.98}, which only achieves arbitrarily small error). 
\end{enumerate}

\subsection{Exact Linear Input \& State Estimator (ELISE)}

For the first variant, we consider the following assumption:
\begin{assump}{$(\mathbf{A1})$}\label {A1}  
We assume that 
\begin{enumerate}[(i)]
\item the noise terms, $w(t) \in \mathbb{R}^q$ and $v(t) \in \mathbb{R}^l$, are mutually uncorrelated, zero-mean, white random signals with known noise statistics: $\mathbb{E} [w(t) w(t')^\top] =Q(t) \delta(t-t')$,  $\mathbb{E} [v(t) v(t')^\top] =R(t) \delta(t-t')$ and $\mathbb{E} [w(t) v(t')^\top] =0$, where $\delta(\cdot)$ is the Dirac delta function, $Q(t) \succeq 0$ and $R(t) \succ 0$ for all $t$. 
\item an additional measurement is available, which contains information about the state derivative $\dot{x}(t)$ and thus, about an equivalent of the `output derivative': 
\begin{align} \label{eq:yoverl}
\begin{array}{ll}
\overline{y}(t) =& \overline{C}(t) \dot{x} (t) +\doverline{C}(t) x(t) + \overline{D}(t) \dot{u}(t)  + \doverline{D}(t) u(t) \\ &+ \overline{H}(t) \dot{d}(t)+\doverline{H}(t) d(t) + \overline{v} (t) \vspace{-0.45cm}
\end{array}
\end{align} \noindent
with the following noise statistics: $\mathbb{E}[\overline{v}(t)]=0$,  $\mathbb{E} [w(t) \overline{v}(t')^\top] =0$, $\mathbb{E} [v(t) \overline{v}(t')^\top] =\grave{R} (t) \delta(t-t')$ and $\mathbb{E} [\overline{v}(t) \overline{v}(t')^\top] =\overline{R}(t) \delta(t-t')$, and where $\overline{R}(t) \succ 0$ and $\grave{R}(t)$ are known. Note that $d(t)$ need not be differentiable because we only make use of $\overline{z}_2:=\overline{T}_2 \overline{y}$ in our filter design and $\overline{T}_2 \overline{H}=0$ ($\overline{T}_2$ is as defined for \eqref{eq:variant1} below).
\end{enumerate}
\end{assump}

Note that the additional measurement $\overline{y}$ is different from the signal $\dot{y}(t)$, which is not well defined due to the derivative of noise. The assumption of an additional measurement is at times reasonable, for e.g., accelerations of mechanical systems are typically measured in addition to state (position and velocity), and sometimes slew rates (rate of change of voltage) in electronics and flow accelerations in fluid systems may also be measured.

However, such availability of an additional measurement can be rare. This is actually the main motivation for considering the second derivative-free variant in the next section. Alternatively, \emph{filtered} derivatives of the output may be used in place of the additional measurement, as is demonstrated to be good enough in the simulation example in Section \ref{sec:reentry}.

With Assumption ($A1$), we consider the following filter:
\begin{align} 
\hspace*{-0.15cm} \hat{d}_1 &=M_1(z_1-C_1 \hat{x}-D_1 u), \label{eq:dhat1}\\
\hspace*{-0.15cm} \nonumber\hat{d}_{2}
 &= M_2( \overline{z}_2-(\overline{C}_2 A+\overline{T}_2 \doverline{C}) \hat{x}-\overline{C}_2 B u-\overline{C}_2 G_1 \hat{d}_1\\
\hspace*{-0.15cm} &\qquad -\overline{D}_2 \dot{u}-T_2 \doverline{D} u), \label{eq:variant1} \\
\hspace*{-0.15cm} \hat{d}  &=V_1 \hat{d}_1+V_2 \hat{d}_2, \label{eq:dhat} \\ 
\hspace*{-0.15cm} \dot{\hat{x}}  &= A \hat{x} + B u + G_1 \hat{d}_{1} +G_2 \hat{d}_2  + L (z_2-C_2 \hat{x}-D_2 u),\hspace{-0.1cm} \label{eq:xhat}
\end{align}
where $\overline{T}_2=\overline{U}_2^\top$ is obtained from the singular value decomposition 
of $\overline{H}=\begin{bmatrix}\overline{U}_{1}& \overline{U}_{2} \end{bmatrix} \begin{bmatrix} \overline{\Sigma} & 0 \\ 0 & 0 \end{bmatrix} \begin{bmatrix} \overline{V}_{1}^{\, \top} \\ \overline{V}_{2}^{\, \top} \end{bmatrix}$, while $\overline{C}_2:=\overline{T}_2 \overline{C}$, $\overline{D}_2:=\overline{T}_2 \overline{D}$ and $\overline{z}_2:=\overline{T}_2 \overline{y}$.  
We also assume that $\overline{T}_2 \doverline{H}=0$, as is the case when $\doverline{H}=0$ or when $\overline{y}$ is the true `output derivative' with $\overline{H}=H$ and $\doverline{H}=\dot{H}$ (by Theorem \ref{thm:Hdot} and Corollary \ref{cor:cor1}). The matrices ${L} \in \mathbb{R}^{n \times (l-p_{H})}$, $M_{1} \in \mathbb{R}^{p_{H} \times p_{H}}$ and $M_{2} \in \mathbb{R}^{(p-p_{H}) \times (l-p_{H})}$ are filter gains that are chosen to minimize the state and input error covariances. 

\begin{algorithm}[!t] \small
\caption{ELISE algorithm}\label{algorithm}
\begin{algorithmic}[1]
\State Initialize: $\hat{x}(t_{0})=\hat{x}_0$; $P^x(t_{0})=\mathcal{P}^x_0$;
\While {$t < t_f$}
\LineComment{Unknown input estimation}
\State $M_1=\Sigma^{-1}$;
\State $\hat{Q}=W Q W^\top+G_1 M_1 R_1 M_1^\top G_1^\top$;
\State $\hat{A}=A-G_1 M_1 C_1$;
\State $\tilde{R}_2=(\overline{C}_2 \hat{A}+\overline{T}_2 \doverline{C}) P^x (\overline{C}_2 \hat{A}+\overline{T}_2 \doverline{C})^\top + \overline{C}_2 \hat{Q} \overline{C}_2^\top+\overline{R}_2$ 
\Statex \hspace{1.5cm} $-\grave{R}_{12}^\top M_1^\top G_1^\top \overline{C}_2^\top-\overline{C}_2 G_1 M_1 \grave{R}_{12}$;
\State $M_2=(G_2^\top \overline{C}_2^\top \tilde{R}_2^{-1} \overline{C}_2 G_2)^{-1} G_2^\top \overline{C}_2^\top \tilde{R}_2^{-1}$;
\State $\hat{d}_1=M_1(z_1-C_1 \hat{x}-D_1 u)$;
\State $\hat{d}_2=M_2( \overline{z}_2-(\overline{C}_2 A+\overline{T}_2 \doverline{C}) \hat{x}-\overline{C}_2 B u-\overline{C}_2 G_1 \hat{d}_1-\overline{D}_2 \dot{u}$ 
\Statex \hspace{1.5cm} $-\overline{T}_2 \doverline{D} u)$;
\State $\hat{d}=V_1 \hat{d}_1+V_2 \hat{d}_2$;
\State $P^d_1=M_1 (C_1 P^x C_1^\top+R_1) M_1^\top$;
\State$P^d_2=(G_2^\top \overline{C}_2^\top \tilde{R}_2^{-1} \overline{C}_2 G_2)^{-1}$;
\State $P^d_{12}=M_1 C_1 P^x (\hat{A}^\top \overline{C}_2^\top+\doverline{C}^\top \overline{T}_2^\top) M_2^\top$ 
\Statex \hspace{1.5cm} $-M_1 R_1 M_1^\top G_1^\top \overline{C}_2^\top M_2^\top+M_1 \grave{R}_{12} M_2^\top$;
\State $P^d=V_1 P^d_1 V_1^\top+ V_1 P^d_{12} V_2^\top + V_2 P^{d \top}_{12} V_1^\top + V_2 P^d_2 V_2^\top$;

\LineComment{State estimation}
\State $\overline{A}=(I-G_2M_2 \overline{C}_2) \hat{A}-G_2 M_2 \overline{T}_2 \doverline{C}$;
\State $\overline{Q}=(I-G_2M_2 \overline{C}_2) \hat{Q} (I-G_2M_2 \overline{C}_2)^\top+ G_2 M_2 \overline{R}_2 M_2^\top G_2^\top$;
\State $L=(P^x C_2^\top -G_2 M_2 \grave{R}_2^\top) R_2^{-1}$;
\State $\dot{\hat{x}}=A \hat{x} + B u + G_1 \hat{d}_{1} +G_2 \hat{d}_2 + L (z_2-C_2 \hat{x} -D_2 u)$;
\State $\dot{P}^x=\overline{A} P^x + P^x \overline{A}^\top +\overline{Q} - L R_2 L^\top$;
\EndWhile
\end{algorithmic}
\end{algorithm}

A summary of the first variant of optimal continuous-time filter is given in Algorithm \ref{algorithm}. The ELISE algorithm has some nice properties, which we will describe here and prove in the Appendix. First, assuming that the filter is uniformly asymptotically stable\footnote{See \cite{Kalman.Bertram.1960} for the definition of uniform asymptotic stability.}, the initial state and unknown input estimate biases are shown to converge exponentially in the following lemmas. Note that the uniform asymptotic stability of the proposed filter will be verified in Theorem \ref{thm:stable}.

\begin{lem}[Convergence of state estimate bias of ELISE] \label{lem:convergence}
Let ELISE be uniformly asymptotically stable. 
Then, its state estimate bias, $\mathbb{E}[\tilde{x}]:=\mathbb{E}[x-\hat{x}]$, decays exponentially, i.e., 
\begin{align}
\| \mathbb{E}[\tilde{x}] \| \leq \beta e^{-\gamma (t-t_0)}, \label{eq:stateError}
\end{align} \noindent
for all $t\geq t_0$, for some constant $\beta$ and $\gamma$. If, in addition, 
$\breve{A}:=\overline{A}-L C_2$ is bounded\footnote{This holds in general, since the system matrices are bounded by assumption and the practical usefulness of a stable filter with unbounded $P^x$ is rather limited.}, 
with an initial state estimate bias given by $\mathbb{E}[\tilde{x}(t_0)]$, then $\beta$ and $\gamma$\footnote{This convergence rate (with $\Gamma=I$) can be shown to be the largest when compared with all bounded $\Gamma$ such that the pair $(\breve{A},\Gamma)$ is uniformly completely observable (see Definition \ref{def:uco}) using an approach similar to \cite[pp. 91-93]{slotine.li.book91}. Note also that in general, the checking of uniform complete controllability or observability is not straightforward. Some classes of systems for which uniform complete controllability can be shown are given in \cite[Section 5]{Silverman.1968}, where $\Gamma=I$ is one such instance.} are given by
\begin{align}
\beta &= \sqrt{\frac{\mathbb{E}[\tilde{x}(t_0)]^\top S(t_0) \mathbb{E}[\tilde{x}(t_0)]}{\underline{\lambda}_{min}(S)}}, \ 
\gamma= \frac{1}{2 \overline{\lambda}_{max}(S)}, \label{eq:betagamma}
\end{align} \noindent
where  
$\overline{\lambda}_{max} (S)$ and $\underline{\lambda}_{min} (S)$ are the supremum and infimum over $t  \geq t_0$ of the largest and smallest eigenvalue of 
$S(t)= \displaystyle \lim_{T \to \infty} \int^{T}_{t} \Phi_{\breve{A}(t)}(t,s) \Phi_{\breve{A}(t)}^\top(t,s) ds \succ 0$ with $ \Phi_{\breve{A}(t)}(\cdot)$ denoting the transition matrix associated with state dynamics $\dot{x}=\breve{A}x$. 
\end{lem}

\begin{lem}[Convergence of unknown input estimate bias of ELISE] \label{lem:convergenceInput}
Let ELISE be uniformly asymptotically stable. 
Then, the input estimate bias of ELISE decays exponentially, i.e., 
there exist $\alpha_1$ and $\gamma$ such that 
\begin{align}
 \|\mathbb{E}[d-\hat{d}]\| := \|\mathbb{E}[\tilde{d}]\|
\leq \alpha_1 e^{-\gamma (t-t_0)}, \label{eq:derror}
\end{align} \noindent
where $\gamma$ is given by \eqref{eq:betagamma} (assuming $\breve{A}$ is bounded as in Lemma \ref{lem:convergence}) and $\alpha_1$ is a positive constant
\begin{align}
\alpha_1&= \beta \sup(\|V_1 M_1 C_1 \|+ \| V_2 M_2 (\overline{C}_2 \hat{A} +\overline{T}_2 \doverline{C})\|). \label{eq:alpha1}
\end{align}
\end{lem}

In addition, the following theorem proves that state and input estimates of ELISE are unbiased and optimal. 
\begin{thm}[Minimum-variance unbiased state and input estimation of ELISE] \label{thm:main}
Suppose ($A1$) holds. If ${\rm rk}(\overline{C}_{2} G_{2})=p-p_{H}$ and $(\overline{A},C_2)$ is detectable, where the matrix $\overline{A}$ is as defined in Algorithm \ref{algorithm}, then the filter gains, $L$, $M_1$ and $M_2$, given in Algorithm \ref{algorithm}, and the differential Riccati equation given by
\begin{align} \label{eq:riccatiEq}
\dot{P}^x=\overline{A} P^x + P^x \overline{A}^\top +\overline{Q} - L R_2 L^\top
\end{align} 
\noindent provide the unbiased, best linear estimate (BLUE) of the unknown input and the minimum-variance unbiased estimate of system states. Moreover, if the optimal filter is uniformly asymptotically stable, the effect of initial state and input estimate bias decays exponentially, as given in \eqref{eq:stateError} and \eqref{eq:derror}.
\end{thm}

However, the optimality of the filter does not guarantee that the filter is stable. Additional assumptions are needed for the uniform asymptotic stability of the filter, similar to the stability requirements of the Kalman-Bucy filter \cite[Theorem 4]{kalman.bucy.1961}.

\begin{thm}[Stability of ELISE] \label{thm:stable}
Using Assumption ($A1$) and the proposed filter, we obtain a `virtual' equivalent system 
\begin{align} \label{eq:eqSys}
\begin{array}{ll}
\dot{x}_e&={A}_e x_e +u_e + {w}_e,\
y_e= C_2 x_e + v_2,
\end{array}
\end{align} \noindent
with $A_e:=\overline{A}-G_2 M_2 \grave{R}_2^\top R_2^{-1}$, $\overline{A}:=(I-G_2 M_2 \overline{C}_2) \hat{A}-G_2 M_2 \overline{T}_2 \doverline{C}$, $u_e\hspace{-0.05cm}=\hspace{-0.05cm}-G_2 M_2 \grave{R}_2^\top R_2^{-1} y_e$, $w_e\hspace{-0.05cm}=\hspace{-0.05cm}G_2M_2 \grave{R}_2^\top R_2^{-1} v_2+ \overline{w}$ and $\overline{w}:=(I-G_2 M_2 \overline{C}_2)Ww-(I-G_2 M_2 \overline{C}_2) G_1 M_1 v_1-G_2 M_2 \overline{v}_2$. If the equivalent system \eqref{eq:eqSys} is
\begin{description}
\item
{(A2)} uniformly completely observable,
\item
 {(A3)} uniformly completely controllable,
\item
{(A4)} $\| {Q}_e \|$ and $\| R_2 \|$ are bounded below and above,
\item
{(A5)} $\| {A}_e \|$ is bounded above,
\end{description}
where the equivalent noise covariances are $\mathbb{E}[w_e(t) w_e^\top(t')]=Q_e(t) \delta(t-t')$, $\mathbb{E}[\overline{w}(t) \overline{w}^\top(t')] =\overline{Q}(t) \delta(t-t')$, $Q_e:=\overline{Q}-G_2 M_2 \grave{R}_2^\top R_2^{-1} \grave{R}_2 M_2^\top G_2^\top$ and $\overline{Q}$  (as defined in Algorithm \ref{algorithm}), then 
the optimal filter given in Algorithm \ref{algorithm} is uniformly asymptotically stable. Moreover, every solution to the variance equation given by the differential Riccati equation, $\dot{P}^x$, in Algorithm \ref{algorithm} starting at $\mathcal{P}^x_0 \succ 0$ converges to a unique $P^x$ as $t \to \infty$.
\end{thm}

Finally, for the time-invariant case, the conditions under which the algebraic Riccati equation of the filter has a unique stationary solution is given by:

\begin{thm}[Convergence to steady-state of ELISE] \label{thm:conv}
Let ${\rm rk}(\overline{C}_{2} G_{2})=p-p_{H}$. Then, in the time-invariant case with $P^x(t_{0}) \succeq 0$, the filter in Algorithm \ref{algorithm} (exponentially) converges to a unique stationary solution if and only if
(i) $({A}_e,C_2)$ is detectable, and 
(ii) $({A}_e,{Q}_e^{\frac{1}{2}})$ is stabilizable 
where matrices ${A}_e$ and ${Q}_e$ are as defined in Theorem \ref{thm:stable}.
\end{thm}

\subsection{Approximate Linear Input \& State Estimator (ALISE)}

For this second variant, we do not assume the availability of an `output derivative', but that such a signal exists. Hence, the ALISE variant uses a special case of the ELISE filter, and the existence of $\overline{y}$ for this special case can be seen as a pseudo-derivative of the output measurement $y$.
\begin{spc}[Special Case of ELISE]
There exists an `output derivative' signal $\overline{y}$ in \eqref{eq:yoverl}, such that $\overline{C}=C$, $\doverline{C}=\dot{C}$, $\overline{D}=D$, $\doverline{D}=\dot{D}$, $\overline{H}=H$ and $\doverline{H}=\dot{H}$; hence, we have $\overline{C}_2=C_2$, $\overline{D}_2=D_2$ and $\overline{T}_2=T_2=U_2^\top$. \label{sp}
\end{spc}

Moreover, the existence of the `output derivative' signal also implies that the derivatives of the input signals $u$ and $d$, as well as the noise signals $w$ and $v$ must exist, which necessitates non-standard noise models and rather strong assumptions on the disturbance signals:

\begin{assump}{$(\mathbf{A1'})$} \label{A2} We assume that
\begin{enumerate}[(i)]
\item the noise signals, $w(t)$ and $v(t)$, are first- and second-order Gauss-Markov (GM) processes, respectively (see, e.g., \cite[pp. 42-47]{gelb.1974} for their properties):
\begin{align}
\begin{array}{rl}
\dot{w}(t)+A_w w(t) &= B_w w_G(t),\\
\ddot{v}(t)+A_{\dot{v}} \dot{v}(t) + A_v v(t) &=B_v v_G(t), \end{array}\label{eq:secondOrder} 
\end{align} \noindent where $w_G(t)$ and $v_G(t)$ are mutually uncorrelated, zero-mean, white noise signals with time-invariant intensities  $Q_G \succeq 0$ and $R_G \succ 0$, respectively. Furthermore, the correlation times of the process and measurement noise are assumed to be short compared
to times of interest.
The second equation in \eqref{eq:secondOrder} is equivalently rewritten as
\begin{align}
\begin{array}{rl}
 \frac{d}{dt} \underline{v}(t) =& \begin{bmatrix} 0 & I \\ - A_v & -A_{\dot{v}} \end{bmatrix} \underline{v}(t)+\begin{bmatrix} 0 \\ B_v \end{bmatrix} v_G(t) \\
 :=&
\underline{A}_v \overline{v}(t)+ \underline{B}_v v_G(t).\end{array} 
\end{align} \noindent
$A_w$, $A_v$ and $A_{\dot{v}}$ are positive semidefinite diagonal matrices, while $w(t_0)$ and $\underline{v}(t_0):=\begin{bmatrix} {v}(t)^\top & \dot{v}(t)^\top \end{bmatrix}^\top$ have known covariance matrices $\mathcal{P}^w_0$ and $\mathcal{P}^v_0$.
For simplicity, we shall assume for the noise models that $-A_w$ and $\underline{A}_v$ are time-invariant and stable, i.e. their eigenvalues are strictly negative, and that $B_w$, $B_v$, $Q_G$ and $R_G$ are also time-invariant and bounded.
\item the inputs $u(t)$ and $d(t)$ are twice and once differentiable, respectively, and that $u(t)$, $\dot{u}(t)$, $\ddot{u}(t)$, ${d}(t)$ and $\dot{d}(t)$ are bounded, as well as that the norm of the system state vectors, matrices and matrix derivatives are bounded.
\end{enumerate}
\end{assump}

In a nutshell, the noise models in Assumption ($A1'$), i.e., Gauss-Markov stochastic noise models, are stochastic processes that satisfy the requirements for both Gaussian processes and Markov processes, and can be viewed as continuous-time analogues of the discrete-time AR(1) and AR(2) processes. The first-order Gauss-Markov process is also known as the Ornstein-Uhlenbeck process, which has been considered in the models of financial mathematics and physical sciences. 
The noise models are specifically chosen such that the signal $\ddot{z}_2(t)$ is well defined for the purpose of analyzing the proposed filter, as is required by Taylor's theorem in \eqref{eq:taylor} of Appendix \ref{sec:errorBound}, and the assumption of short correlation times is such that the noise terms are not colored\footnote{Note that we do not attempt to solve the estimation problem with colored noise, which is a subject of future research, as this would require the development of state and unknown input filters for systems with correlated noise terms and moreover, output derivatives would need to be computed, as is pointed out in \cite{Bryson.65}.}.
The covariance matrices of the noise models can either be determined in experiments, or simply chosen as tuning parameters, which is commonplace in practice.

The assumption of bounded derivatives of $d$ is also rather strong, but is unfortunately necessary for a meaningful analysis of the input and state filtering problem. 
However, this assumption may actually not be needed in practice, as evidenced by our example in Section \ref{sec:examples} with a non-smooth disturbance. 

For this case, we now propose the following filter: 
\begin{align}
\nonumber \hat{d}_1  &=M_1(z_1-C_1 \hat{x}-D_1 u),\\ 
\nonumber \hat{d}_{2} &= M_2( \frac{z_2(t)-z_2(t-\mathfrak{d}t)}{\mathfrak{d}t}-(C_2 A+T_2 \dot{C}) \hat{x}  \\ & \qquad-C_2 B u-C_2 G_1 \hat{d}_1-D_2 \dot{u}-T_2 \dot{D} u), \label{eq:variant2}\\
\nonumber \hat{d}  &=V_1 \hat{d}_1+V_2 \hat{d}_2, \\  
\nonumber \dot{\theta} &= (\overline{A}-L C_2)(G_2 M_2 z_2 - G_2 M_2 D_2 u + \theta)  \\ 
\nonumber &\qquad  + (\overline{B}\hspace{-0.05cm}-\hspace{-0.05cm}LD_2)u + \overline{G} M_1 z_1 + L z_2\hspace{-0.05cm}-\hspace{-0.05cm} \dot{\Phi}_1 y-\dot{\Phi}_2 u,\\
\hat{x}&=G_2 M_2 z_2 - G_2 M_2 D_2 u + \theta, \label{eq:xhat2}
\end{align} 

\noindent where $\overline{A}$, $\overline{B}$, $\overline{G}$, $\dot{\Phi}_1$ and $\dot{\Phi}_2$ are as defined in Algorithm \ref{algorithm2}, the matrices corresponding to the Special Case \ref{sp} are used in \eqref{eq:variant2} as well as in Algorithm \ref{algorithm2} and the matrices ${L} \in \mathbb{R}^{n \times (l-p_{H})}$, $M_{1} \in \mathbb{R}^{p_{H} \times p_{H}}$ and $M_{2} \in \mathbb{R}^{(p-p_{H}) \times (l-p_{H})}$ are filter gains. 
Note that the output derivatives $\dot{y}$ is essentially obtained by finite difference approximation, $\dot{y} \approx \frac{y(t)-y(t-\mathfrak{d}t)}{\mathfrak{d}t}$, where $\mathfrak{d}t$ can be chosen arbitrarily.

\begin{algorithm}[!t] \footnotesize
\caption{ALISE algorithm}\label{algorithm2}
\begin{algorithmic}[1]
\State Initialize: $P^x(t_{0})=\mathcal{P}^x_0$; $P^w(t_{0})=\mathcal{P}^w_0$;  $P^v(t_{0})=\mathcal{P}^v_0$; $\theta(t_0)=\hat{x}_0-G_2 M_2 z_2(t_0) + G_2 M_2 D_2 u(t_0) $; etc.;
 \While {$t < t_f$}
 \LineComment{State estimation}
 \State $\hat{x}=G_2 M_2 z_2 - G_2 M_2 D_2 u + \theta$;
 \LineComment{Unknown input estimation}
\State $\dot{P}^w=-A_w P^w -P^w A_w^\top + B_w Q_G B_w^\top $;
\State $\dot{P}^v=\overline{A}_v P^v +P^v \overline{A}_v^\top + \overline{B}_v R_G \overline{B}_G^\top$;
 \State $R_{1}= \begin{bmatrix} T_1 & 0 \end{bmatrix} P^v \begin{bmatrix} T_1 & 0 \end{bmatrix}^\top $; 
 \State $R_{2}= \begin{bmatrix} T_2 & 0 \end{bmatrix} P^v \begin{bmatrix} T_2 & 0 \end{bmatrix}^\top $;
 \State $\overline{R}_2= \begin{bmatrix} 0 & T_2 \end{bmatrix} P^v \begin{bmatrix} 0 & T_2 \end{bmatrix}^\top $;  
 \State $\grave{R}_{2}= \begin{bmatrix} T_2 & 0 \end{bmatrix} P^v \begin{bmatrix} 0 & T_2 \end{bmatrix}^\top $; 
 \State $\grave{R}_{12}= \begin{bmatrix} T_1 & 0 \end{bmatrix} P^v \begin{bmatrix} 0 & T_2 \end{bmatrix}^\top $; 
\State  $\hat{A}=A-G_1 M_1 C_1$; 
\State $\hat{Q}=W P^w W^\top+G_1 M_1 R_1 M_1^\top G_1^\top$; 
\State $M_1=\Sigma^{-1}$;
\State $\tilde{R}_2=(C_2 \hat{A}+T_2 \dot{C}) P^x (C_2 \hat{A}+T_2 \dot{C})^\top + C_2 \hat{Q} C_2^\top+\overline{R}_2 $ 
\Statex \hspace{1.5cm} $- \grave{R}_{12}^\top M_1^\top G_1^\top C_2^\top - C_2 G_1 M_1 \grave{R}_{12}$;
\State $M_2=(G_2^\top C_2^\top \tilde{R}_2^{-1} C_2 G_2)^{-1} G_2^\top C_2^\top \tilde{R}_2^{-1}$;
\State $\hat{d}_1=M_1(z_1-C_1 \hat{x}-D_1 u)$;
\State $\overline{d}_2=M_2( \frac{z_2(t)-z_2(t-\mathfrak{d}t)}{\mathfrak{d}t}-(C_2 A+T_2 \dot{C}) \hat{x}-C_2 B u-C_2 G_1 \hat{d}_1$ 
\Statex \hspace{1.5cm} $-D_2 \dot{u}-T_2 \dot{D} u)$;
\State $\hat{d}=V_1 \hat{d}_1+V_2 \overline{d}_2$;
\State $P^d_1=M_1 (C_1 P^x C_1^\top+R_1) M_1^\top$;
\State$P^d_2 \approx(G_2^\top C_2^\top \tilde{R}_2^{-1} C_2 G_2)^{-1}$;
\State $P^d_{12} \approx M_1 C_1 P^x (\hat{A}^\top C_2^\top+\dot{C}^\top T_2^\top) M_2^\top$ 
\Statex \hspace{1.5cm} $-M_1 R_1 M_1^\top G_1^\top C_2^\top M_2^\top + M_1 \grave{R}_{12} M_2^\top$;
\State $P^d \approx V_1 P^d_1 V_1^\top+ V_1 P^d_{12} V_2^\top + V_2 P^{d \top}_{12} V_1^\top + V_2 P^d_2 V_2^\top$;

\LineComment{State estimation}
\State $\overline{A}=(I-G_2M_2 C_2) \hat{A}-G_2 M_2 T_2 \dot{C}$; 
\State $\overline{B}=(I-G_2M_2 C_2) (B-G_1 M_1 D_1)-G_2 M_2 T_2 \dot{D}$;
\State $\overline{G}=(I-G_2M_2 C_2) G_1$; $\overline{Q}=(I-G_2M_2 C_2) \hat{Q} (I-G_2M_2 C_2)^\top$;
\State $L=(P^x C_2^\top -G_2 M_2 \grave{R}_2^\top) R_2^{-1}$;
\State $\dot{P}^x=\overline{A} P^x + P^x \overline{A}^\top +\overline{Q} - L R_2 L^\top$;
\State $E$, $F$ and $\dot{\Sigma}$ according to Section \ref{sec:SVD}, e.g., \eqref{eq:EF}; 
\State $\dot{T}_1=U_1^\top R U_2 (U_2^\top R U_2)^{-1} (U_2^\top \dot{R} U_2) (U_2^\top R U_2)^{-1}+E^\top U_1^\top$ 
\Statex \hspace{1.5cm} $-U_1^\top \dot{R} U_2 (U_2^\top R U_2)^{-1}- E^\top U_1^\top R U_2 (U_2^\top R U_2)^{-1}$;
\State $\dot{M}_1=-\Sigma^{-1} \dot{\Sigma} \Sigma^{-1}$; 
\State $\dot{\grave{R}}_{12}=\begin{bmatrix} \dot{T}_1 & 0 \end{bmatrix} P^v \begin{bmatrix} 0 & T_2 \end{bmatrix}^\top+\begin{bmatrix} T_1 & 0 \end{bmatrix} \dot{P}^v \begin{bmatrix} 0 & T_2 \end{bmatrix}^\top$;
\State $\dot{R}_1=\begin{bmatrix} \dot{T}_1 & 0 \end{bmatrix} P^v \begin{bmatrix} {T}_1 & 0 \end{bmatrix}^\top + \begin{bmatrix} {T}_1 & 0 \end{bmatrix} \dot{P}^v \begin{bmatrix} {T}_1 & 0 \end{bmatrix}^\top+ \begin{bmatrix} {T}_1 & 0 \end{bmatrix} P^v \begin{bmatrix} \dot{T}_1 & 0 \end{bmatrix}^\top$;
\State $\dot{\hat{A}}=\dot{A} - (\dot{G} V_1 + G V_1 F) M_1 C_1 - G_1 \dot{M}_1 C_1 - G_1 M_1 (T_1 \dot{C} + \dot{T}_1 C)$;
\State $\dot{\hat{Q}}=(\dot{G} V_1 + G V_1 F) M_1 R_1 M_1^\top G_1^\top + G_1 \dot{M}_1 R_1 M_1^\top G_1^\top $ 
\Statex \hspace{1.5cm} $+W \dot{P^w} W^\top + \dot{W} P^w W^\top + W P^w \dot{W} +G_1 M_1 R_1 \dot{M}_1^\top G_1^\top$
\Statex \hspace{1.5cm} $+G_1 M_1 R_1 M_1^\top \dot{G}_1^\top+G_1 M_1 \dot{R}_1 M_1^\top G_1^\top $;

\State $\dot{\tilde{R}}_2=(T_2 \dot{C} \hat{A} + C_2 \dot{\hat{A}} + T_2 \ddot{C})P^x (C_2 \hat{A} +T_2 \dot{C})^\top+(C_2 \hat{A}$
\Statex \hspace{1.5cm} $+T_2 \dot{C}) \dot{P}^x (C_2 \hat{A} +T_2 \dot{C})^\top+(C_2 \hat{A} +T_2 \dot{C})P^x(T_2 \dot{C} \hat{A} $
\Statex \hspace{1.5cm} $+ C_2 \dot{\hat{A}}+ T_2 \ddot{C})^\top + C_2 \dot{\hat{Q}}C_2^\top + T_2 \dot{C} \hat{Q} C_2^\top+C_2 \hat{Q} \dot{C}^\top T_2^\top$
\Statex \hspace{1.5cm} $+ \begin{bmatrix} 0 & T_2\end{bmatrix} \dot{P}^v \begin{bmatrix} 0 & T_2\end{bmatrix}^\top -\dot{\grave{R}}_{12}^\top M_1^\top G_1^\top C_2^\top-{\grave{R}}_{12}^\top \dot{M}_1^\top G_1^\top C_2^\top$
\Statex \hspace{1.5cm} $-{\grave{R}}_{12}^\top {M}_1^\top (V_1^\top \dot{G}^\top  + F^\top V_1^\top G^\top) C_2^\top$
\Statex \hspace{1.5cm} $-{\grave{R}}_{12}^\top {M}_1^\top G_1^\top \dot{C}^\top T_2^\top -C_2 G_1 M_1 \dot{\grave{R}}_{12}-C_2 G_1 \dot{M}_1 {\grave{R}}_{12}$
\Statex \hspace{1.5cm} $-C_2 (\dot{G} V_1 + G V_1 F) M_1 {\grave{R}}_{12}-T_2 \dot{C} G_1 M_1 {\grave{R}}_{12}$;

\State $\dot{M}_2=P^d_2(-G_2^\top C_2^\top \tilde{R}_2^{-1} \dot{\tilde{R}}_2 \tilde{R}_2^{-1}+V_2^\top \dot{G}^\top C_2^\top \tilde{R}_2^{-1} $ 
\Statex \hspace{1.5cm} $+G_2^\top \dot{C}^\top U_2 \tilde{R}_2^{-1})-P^d_2 (- G_2^\top C_2^\top \tilde{R}_2^{-1} \dot{\tilde{R}}_2 \tilde{R}_2^{-1} C_2 G_2$
\Statex \hspace{1.5cm} $+V_2^\top \dot{G}^T C_2^\top \tilde{R}_2^{-1} C_2 G_2 + G_2^\top \dot{C} ^\top U_2 \tilde{R}_2^{-1} C_2 G_2 $ 
\Statex \hspace{1.5cm} $+ G_2^\top C_2^\top \tilde{R}_2^{-1} U_2^\top \dot{C} G_2+G_2^\top C_2^\top \tilde{R}_2^{-1} C_2 \dot{G} V_2)  M_2$;
\State $\dot{\Phi}_1=\dot{G} V_2 M_2 U_2^\top + G_2 \dot{M_2} U_2^\top$; 
\State $\dot{\Phi}_2=-\dot{G} V_2 M_2 D_2^\top - G_2 \dot{M_2} D_2-G_2 M_2 U_2^\top \dot{D}$;
\State $\dot{\theta} = (\overline{A}-L C_2) (G_2 M_2 z_2 - G_2 M_2 D_2 u + \theta)+ (\overline{B}-LD_2)u $ 
\Statex \hspace{1.5cm} $+ \overline{G} M_1 z_1 + L z_2-\dot{\Phi}_1 y - \dot{\Phi}_2 u$;
\EndWhile
\end{algorithmic}
\end{algorithm} \normalsize

The ALISE variant is summarized in Algorithm \ref{algorithm2} and has some nice properties that are described here and proven in the Appendix. Similar to the ELISE variant, assuming that the filter is uniformly asymptotically stable\footnote{See \cite{Kalman.Bertram.1960} for the definition of uniform asymptotic stability.} (verified later in Theorem \ref{thm:stable2}), the initial state and unknown input estimate biases converge exponentially. 

\begin{lem}[Convergence of state estimate bias of ALISE] \label{lem:convergence2}
Let ALISE be uniformly asymptotically stable 
and $\breve{A}:=\overline{A}-L C_2$ be bounded. Then, the state estimate bias, $\mathbb{E}[\tilde{x}]:=\mathbb{E}[x-\hat{x}]$, decays exponentially, as in \eqref{eq:stateError} and \eqref{eq:betagamma}. 
\end{lem}

\begin{lem}[Convergence of unknown input estimate bias of ALISE] \label{lem:convergenceInput2}
Let ALISE be uniformly asymptotically stable. 
Then, 
the unknown  input estimate convergence properties for ALISE (with matrices as given in Special Case \ref{sp}) 
are given by
\begin{align}
 \|\mathbb{E}[d-\hat{d}]\| := \|\mathbb{E}[\tilde{d}]\| &\leq \alpha_1 e^{-\gamma (t-t_0)} + \alpha_2 \mathfrak{d}t, \label{eq:inputErrALISE}\\
 |{\rm tr}(P^{d} -P^{\overline{d}})|  &\leq  \alpha_3 (\mathfrak{d}t)^2, \label{eq:traceDiff}
\end{align}

\noindent with $\alpha_1$ given in \eqref{eq:alpha1}, as well as $\alpha_2$ and $\alpha_3$ given by \begin{align}
\alpha_2=  \frac{1}{2} \sup \| V_2 M_2 \mathbb{E}[\ddot{z}_2]\| , \ \alpha_3= \sup | \zeta |, \label{eq:alpha23}
\end{align}
with $\mathbb{E}[\ddot{z}_2]$ and $\zeta$ given by
\begin{align}
&\hspace{-0.2cm}\begin{array}{rl}
\mathbb{E}[\ddot{z}_2] &=T_2[(2 \dot{C} A\hspace{-0.05cm} +\hspace{-0.05cm}C \dot{A}\hspace{-0.05cm}+\hspace{-0.05cm}C A^2 \hspace{-0.05cm}+\hspace{-0.05cm} \ddot{C})x +(\dot{C} B\hspace{-0.05cm} + \hspace{-0.05cm}C A B\\ &+\hspace{-0.05cm} C \dot{B} \hspace{-0.05cm}+\hspace{-0.05cm} \dot{C} B \hspace{-0.05cm}+\hspace{-0.05cm} \ddot{D})u + (CB \hspace{-0.1cm}+\hspace{-0.1cm} 2 \dot{D}) \dot{u} + D \ddot{u} +(2 C \dot{G}_1 \\ &
+ CA G_1 + C \dot{G}_2) d_1 +C G_2 \dot{d}_1
+(2 \dot{C} G_2 +CAG_2 \\
& + C \dot{G}_2) d_2 + C G_2 \dot{d}_2],
\end{array}\hspace{-0.1cm} 
\end{align}
\begin{align}
&\begin{array}{rl}
\hspace{-0.1cm}\zeta&=\big|\frac{1}{4} {\rm tr}(M_2 T_2(C W B_w Q_G B_w^\top  W^\top C^\top +\begin{bmatrix} -A_v & -A_{\dot{v}} \end{bmatrix} \\
&\quad P^v \begin{bmatrix} -A_v & -A_{\dot{v}} \end{bmatrix}^\top + B_v R_G B_v^\top +(\dot{C}W + C \dot{W} \\
  &  \ - C WA_w)P^w(\dot{C}W + C \dot{W} - C WA_w)^\top )  T_2^\top M_2^\top)\big|,\hspace{-0.1cm}
 \end{array}
\end{align} \noindent
assuming that $x$ is bounded, whereas $P^w$ and $P^v$ are bounded as a result of Assumption  ($A1'$)\footnote{With the assumptions in Assumption ($A1'$), $P^w$ and $P^v$ are bounded and their bounds can be found in \cite{Kalman.1960,Anderson.1967}.}. $P^{d}$ is the error covariance matrix of ALISE, $P^{\overline{d}}$ is the error covariance matrix of the best linear unbiased input estimate assuming direct access to $\dot{y}$, and $\mathfrak{d}t$ can be chosen to be arbitrarily small. 
\end{lem}

The next theorem shows that state estimate of ALISE is unbiased and optimal, but the input estimate of ALISE is only approximately unbiased to any precision. 

\begin{thm}[Minimum-variance unbiased state estimation of ALISE] \label{thm:main2}
Suppose ($A1'$) holds. If ${\rm rk}(C_{2} G_{2})=p-p_{H}$ and $(\overline{A},C_2)$ is detectable, where the matrix $\overline{A}$ is as defined in Algorithm \ref{algorithm2}, then the filter gains, $L$, $M_1$ and $M_2$, given in Algorithm \ref{algorithm2} along with the differential Riccati equation given in \eqref{eq:riccatiEq} provide the unbiased minimum-variance unbiased estimate of system states. 

Moreover, if the filter is uniformly asymptotically stable, ALISE 
satisfies the error bounds for state estimate bias in \eqref{eq:stateError}, and its bound on the unknown input estimate bias due to initial state estimate bias and approximation errors is given by \eqref{eq:inputErrALISE} and \eqref{eq:traceDiff}.
\end{thm}

Once again, the optimality of the ALISE filter does not guarantee that the filter is stable. Additional assumptions are needed for the uniform asymptotic stability of the filter. 

\begin{thm}[Stability of ALISE] \label{thm:stable2}
Using Assumption ($A1'$) and the proposed filter, we obtain a `virtual' equivalent system \eqref{eq:eqSys} with matrices as given in Special Case \ref{sp}. If the equivalent system satisfies Assumptions ($A2$), ($A3$), ($A4$) and ($A5$) in Theorem \ref{thm:stable}, then the optimal filter given in Algorithm \ref{algorithm2} is uniformly asymptotically stable. Moreover, every solution to the differential Riccati equation, $\dot{P}^x$, in Algorithm \ref{algorithm2} starting at  $\mathcal{P}^x_0 \succ 0$ converges to a unique $P^x$ as $t \to \infty$.
\end{thm}

Finally, for the time-invariant case, the conditions under which the algebraic Riccati equation of the filter has a unique stationary solution is given by:

\begin{thm}[Convergence to steady-state of ALISE] \label{thm:conv2}
Let ${\rm rk}(\overline{C}_{2} G_{2})=p-p_{H}$. Then, in the time-invariant case with $P^x(t_{0}) \succeq 0$, the filter in Algorithm \ref{algorithm2} (exponentially) converges to a unique stationary solution if and only if
(i) $({A}_e,C_2)$ is detectable, and 
(ii) $({A}_e,{Q}_e^{\frac{1}{2}})$ is stabilizable 
where matrices ${A}_e$ and ${Q}_e$ are as defined in Theorem \ref{thm:stable}.
\end{thm}

\begin{prop} \label{remark} For the Special Case \ref{sp}, a system property known as strong observability\footnote{That is, the condition under which the initial condition $x_0$ and the unknown input signal history, $d(\tau)$ for all $0 \leq \tau \leq t$ can be uniquely determined from the measured output history $y(\tau)$ for all $0 \leq \tau \leq t$ (see, e.g., \cite{Hautus.1983})} implies that the pair $({A}_e,C_2)$ is observable; and that $C_2$ and $G_2$ have full rank. A full-rank $G_2$ is a necessary condition for ${\textrm rank}(C_2 G_2)=p-p_H$, while $C_2$ with full rank is also necessary if $l=p$. Hence, strong observability is closely related to the fact that a minimum-variance unbiased estimator exists and admits a steady-state solution. A similar condition also holds for the optimal discrete-time filter in \cite{Yong.Zhu.ea.Automatica16}.
\end{prop}

\section{Separation Principle \& Disturbance Rejection}\label{sec:sepPrinciple}

We now investigate the stability of the closed-loop system, when the controller is a state feedback controller with disturbance rejection terms, where the true state and unknown input are replaced by their estimated values (cf. previous section): 
\begin{align} \label{eq:control}
u=-K \hat{x} - J_1 \hat{d}_1 - J_2 \hat{d}_2,
\end{align} \noindent
where $K$ is the state feedback gain, while $J_1$ and $J_2$ are the disturbance rejection gains. 

The following theorem shows that there also exists a separation principle for linear stochastic systems with unknown inputs, i.e., the designs of the state and input feedback controller and estimator can be carried out independently.
\begin{thm}(Separation Principle) \label{thm:separation}
The state feedback controller gain $K$ in \eqref{eq:control} can be designed independently of the state and input estimator gains $L$, $M_1$ and $M_2$ in ELISE and ALISE (cf. Algorithms \ref{algorithm} and \ref{algorithm2}).
\end{thm}
\begin{proof}
Substituting \eqref{eq:control} into \eqref{eq:sys} and from \eqref{eq:xtilde}, \eqref{eq:dtilde} and \eqref{eq:d} (for ALISE, with the matrices for Special Case \ref{sp}), we have 

\vspace{-0.3cm}
{\small 
\begin{align*} 
\begin{array}{ll} \begin{bmatrix} \dot{x} \\ \dot{\tilde{x}} \end{bmatrix} &= \begin{bmatrix} A-BK & B(K-J_1 M_1 C_1   -J_2 M_2 (\overline{C}_2 \hat{A}+\overline{T}_2 \doverline{C})) \\ 0 & \overline{A}-L C_2 \end{bmatrix} \begin{bmatrix} {x} \\ {\tilde{x}} \end{bmatrix}   \\&+ \begin{bmatrix} G_1-BJ_1 & G_2-BJ_2 \\ 0 & 0 \end{bmatrix} \begin{bmatrix} d_1 \\ d_2 \end{bmatrix} \\
&  + \begin{bmatrix}  (I -BJ_2 M_2 \overline{C}_2 ) W  & 0 &  \begin{array}{c} BJ_2 M_2 \overline{C}_2 G_1 M_1\\ -BJ_1 M_1\end{array}  & 0 & \hspace{-0.1cm}-BJ_2 M_2 \\ 0 & I & 0 & -L & 0 \end{bmatrix} \mathbf{w},\end{array}
\end{align*}}\normalsize \noindent
where $\mathbf{w}:=\begin{bmatrix} w^\top & \overline{w}^\top & v_1^\top & v_2^\top & \overline{v}_2^\top \end{bmatrix}^\top$.  Since the state matrix has a block diagonal structure, their eigenvalues are given by
\begin{align*}
\det(\lambda I-A +BK) \det(\lambda I-\overline{A} +LC_2) = 0.
\end{align*} \noindent
It can thus be seen that the eigenvalues of the controller and estimator
are independent of each other. 
\end{proof}

Hence, the state feedback gain, $K$, can be independently designed 
 (e.g., with Linear Quadratic Regulator (LQR)) with no effect on the stability of the estimator (ELISE or ALISE). 
Moreover, $J_1$ and $J_2$ can be chosen such that the effect of disturbance input on the closed loop system is reduced. For instance, we can minimize the induced 2-norms of $G_1-B J_1$ and $G_2-B J_2$, which are semidefinite programs\footnote{Semidefinite programs are convex optimization problems for which software packages, e.g. CVX \cite{cvx,Grant.08}, 
are available.} ($i=1,2$): 
\begin{align*}
{\rm minimize\quad  } &\gamma_i \\
{\rm subject \ to \ }  &\begin{bmatrix} \gamma_i I & G_i-B J_i \\  (G_i-B J_i)^\top & \gamma_i I \end{bmatrix} \succeq 0. 
\end{align*} 

In addition, $J_1$ and $J_2$ must also be chosen so that $u$, $\hat{d}_1$ and $\hat{d}_2$ can be uniquely determined. First, $\hat{d}_1$ and $\hat{d}_2$ become implicit equations; thus, the choices of $J_1$ and $J_2$ must be such that $\tilde{J}:=\begin{bmatrix} I-M_1 D_1 J_1 & -M_1 D_1 J_2 \\ M_2 (\overline{C}_2 G_1 -(\overline{C}_2 B + \overline{T}_2 \doverline{D})J_1) & I-M_2 (\overline{C}_2 B+\overline{T}_2 \doverline{D}) J_2 \end{bmatrix}$ is invertible. The explicit expressions for $\hat{d}_1$ and $\hat{d}_2$ in ELISE (Algorithm \ref{algorithm}) with 
$\tilde{J}^{-1}:=\begin{bmatrix} \tilde{J}^\circ_{11} & \tilde{J}^\circ_{12} \\ \tilde{J}^\circ_{21} & \tilde{J}^\circ_{22} \end{bmatrix}$ are

{
\begin{align} 
\label{eq:dexp}
\begin{array}{ll}
\hat{d}_1&=\tilde{J}^\circ_{11} M_1 z_1+\tilde{J}^\circ_{12}M_2 \overline{z}_2 - [\tilde{J}^\circ_{11} (C_1-D_1K)+\tilde{J}^\circ_{12}(\overline{C}_2 A\\
& \quad +\overline{T}_2 \doverline{C} - (\overline{C}_2  B + \overline{T}_2 \doverline{D})K)] \hat{x}- \tilde{J}^\circ_{12} M_2 \overline{D}_2 \dot{u},\\
\hat{d}_2&=\tilde{J}^\circ_{21} M_1 z_1+\tilde{J}^\circ_{22}M_2 \overline{z}_2 - [\tilde{J}^\circ_{21} (C_1-D_1K)+\tilde{J}^\circ_{22}(\overline{C}_2 A\\
& \quad +\overline{T}_2 \doverline{C} - (\overline{C}_2  B + \overline{T}_2 \doverline{D})K)] \hat{x}- \tilde{J}^\circ_{22} M_2 \overline{D}_2 \dot{u}. \vspace{-0.45cm}
\end{array}
\end{align}\normalsize
Substituting \eqref{eq:dexp} back into \eqref{eq:control}, we obtain
\begin{align} \label{eq:controlexp}
\hspace{-0.2cm}\begin{array}{rl}
u&=[(J_1\tilde{J}^\circ_{11}+J_2\tilde{J}^\circ_{21}) (C_1-D_1K)+(J_1 \tilde{J}^\circ_{12}+J_2\tilde{J}^\circ_{22})\\
&\ (\overline{C}_2 A +\overline{T}_2 \doverline{C} - (\overline{C}_2  B + \overline{T}_2 \doverline{D})K) -K] \hat{x}\\ &
\ -(J_1\tilde{J}^\circ_{11}+J_2\tilde{J}^\circ_{21})M_1 z_1 -(J_1 \tilde{J}^\circ_{12}+J_2\tilde{J}^\circ_{22})M_2 \overline{z}_2\\
&\ + (J_1 \tilde{J}^\circ_{12}+J_2\tilde{J}^\circ_{22})M_2 \overline{D}_2 \dot{u},
\end{array}
\end{align} \noindent
which is an ordinary differential equation for $u$ if $ (J_1 \tilde{J}^\circ_{12}+J_2\tilde{J}^\circ_{22})M_2 \overline{D}_2$ is invertible. Notice that if $\overline{D}_2=0$, then $u$, $\hat{d}_1$ and $\hat{d}_2$ can be directly obtained. Otherwise, we need $\begin{bmatrix} J_1 & J_2 \end{bmatrix}$ to have full column rank (hence at least as many control inputs as disturbance inputs , i.e., $m \geq p$) such that $\hat{d}_1$ and $\hat{d}_2$ can be uniquely determined by
$\begin{bmatrix} \hat{d}_1 \\ \hat{d}_2 \end{bmatrix} =\begin{bmatrix} J_2 & J_2 \end{bmatrix}^\dagger (-u-K \hat{x})$. 
To extend the above explicit equations for $\hat{d}_1$, $\hat{d}_2$ and $u$ in \eqref{eq:dexp} and \eqref{eq:controlexp} to the ALISE algorithm, we use the matrices of the Special Case \ref{sp} and substitute $\overline{z}_2$ with $\frac{z_2(t)-z_2(t-\mathfrak{d}t)}{\mathfrak{d}t}$, as well as $\hat{x}$ with  the estimator state $\theta$ given in \eqref{eq:xhat2}. 

Finally, note that if the system in \eqref{eq:system} fulfills a \emph{matching condition}\footnote{The matching condition assumption is common for disturbance rejection in the sliding mode and adaptive control literature.}, i.e. , 
$\exists J$ such that $u=Jd$ and $BJd=Gd$, the above minimization procedure will exactly cancel out the disturbance input.

\section{Illustrative Examples} \label{sec:examples}

To illustrate the effectiveness of the proposed filters, we consider two examples. The first is a nonlinear vehicle reentry problem that demonstrates that our formulation is suitable for linearization-based nonlinear filtering, and the latter example of helicopter hover control allows us to discuss the performance of our filters in the absence of linearization effects.

\subsection{Nonlinear Vehicle Reentry Problem} \label{sec:reentry}

We first consider an example with a vehicle that enters the atmosphere at high altitude and a very high speed, with 
nonlinear vehicle dynamics (based on  \cite{Julier.2004}, cf. Fig. \ref{fig:spaceX}):
\begin{align} \label{eq:example}
\hspace{-0.5cm}\begin{array}{ll}
\dot{x}_1 (t) &= x_3 (t),\\
\dot{x}_2 (t) &= x_4 (t),\\
\dot{x}_3 (t) &= \mathcal{D}(t) x_3(t) + \mathcal{G}(t) x_1 (t) +u_1(t) + w_1(t),\\
\dot{x}_4 (t) &= \mathcal{D}(t) x_4(t) + \mathcal{G}(t) x_2 (t) +u_2(t) + d^w_1 (t) + w_2(t),\\
\dot{x}_5 (t) &= 0,
\end{array}\hspace{-0.55cm}
\end{align}
where  $x_1(t)$ and $x_3(t)$ are the vertical position and velocity of the body, $x_2(t)$ and $x_4(t)$ are the horizontal position and velocity and $x_5(t)$ is
an unknown aerodynamic parameter of the vehicle. $d^w_1(t)$ denotes horizontal disturbance crosswinds that we assume is unknown, whereas $w(t):=[w_1(t), w_2(t)]^\top$ is the process noise. 
The drag-related force term, $\mathcal{D}(t)$, and the gravity-related force term, $\mathcal{G}(t)$, are given by
\begin{align*}
\begin{array}{rl}
\mathcal{D}(t) &= -\beta_0 e^{x_5(t)} e^{\frac{R_0-\sqrt{x_1(t)^2+x_2(t)^2}}{H_0}} \sqrt{x_3(t)^2+x_4(t)^2},\\
\mathcal{G}(t)&=-\frac{G m_0}{(\sqrt{x_1(t)^2+x_2(t)^2})^3},
\end{array}
\end{align*}
with $\beta_0=-0.59783$, $H_0=13.406$, $Gm_0=3.986\times10^5$ and $R_0=6374$.
The motion of the vehicle is measured by a radar that is located at $(x_r,y_r)$. It is able to measure range, bearing and range rate
\begin{align*}
\begin{array}{rl}
y_1(t)&=\sqrt{(x_1(t)-x_r)^2+(x_2(t)-y_r)^2} +d^e_2(t) +v_1 (t)\\
& \hspace{-0.1cm}:=h_1(t) +d^e_2(t) +v_1 (t),\\
y_2(t)&=\arctan \bigg(\frac{x_2(t)-y_r}{x_1(t)-z_r}\bigg) +v_2(t),\\
y_3(t)&=\frac{(x_1(t)-z_r) x_3(t) + (x_2(t)-y_r) x_4(t)}{\sqrt{(x_1(t)-x_r)^2+(x_2(t)-y_r)^2}}+v_3(t)\\
&=\dot{h}_1(t)+v_3(t),
\end{array}
\end{align*} \noindent
where $d^e_2(t)$ denotes an unknown measurement error/fault, whereas $v(t):=[v_1(t),  v_2(t), v_3(t)]^\top$ is the measurement noise. 
Since both the system dynamics and measurements are nonlinear, we consider their linearization about a given reference trajectory to obtain a time-varying linear system. 
In this example, the chosen reference trajectory consists of polynomials $x_1(t)=\sum_{i=0}^3 a_i t^i$ and $x_2(t)=\sum_{i=0}^3 b_i t^i$ and $x_5(t)=c_0$, where the coefficients are chosen to bring the vehicle from the initial reference state $x_r(0)=[6500.4,-1.8093,349.14,-6.7967,0.7]^\top$ to the final state $x_r(t_f)=[6400,-0.5,150,-0.5,0.7]^\top$ in $t_f=200s$.

\begin{figure}[!ttp]
\begin{center}
\includegraphics[scale=0.125]{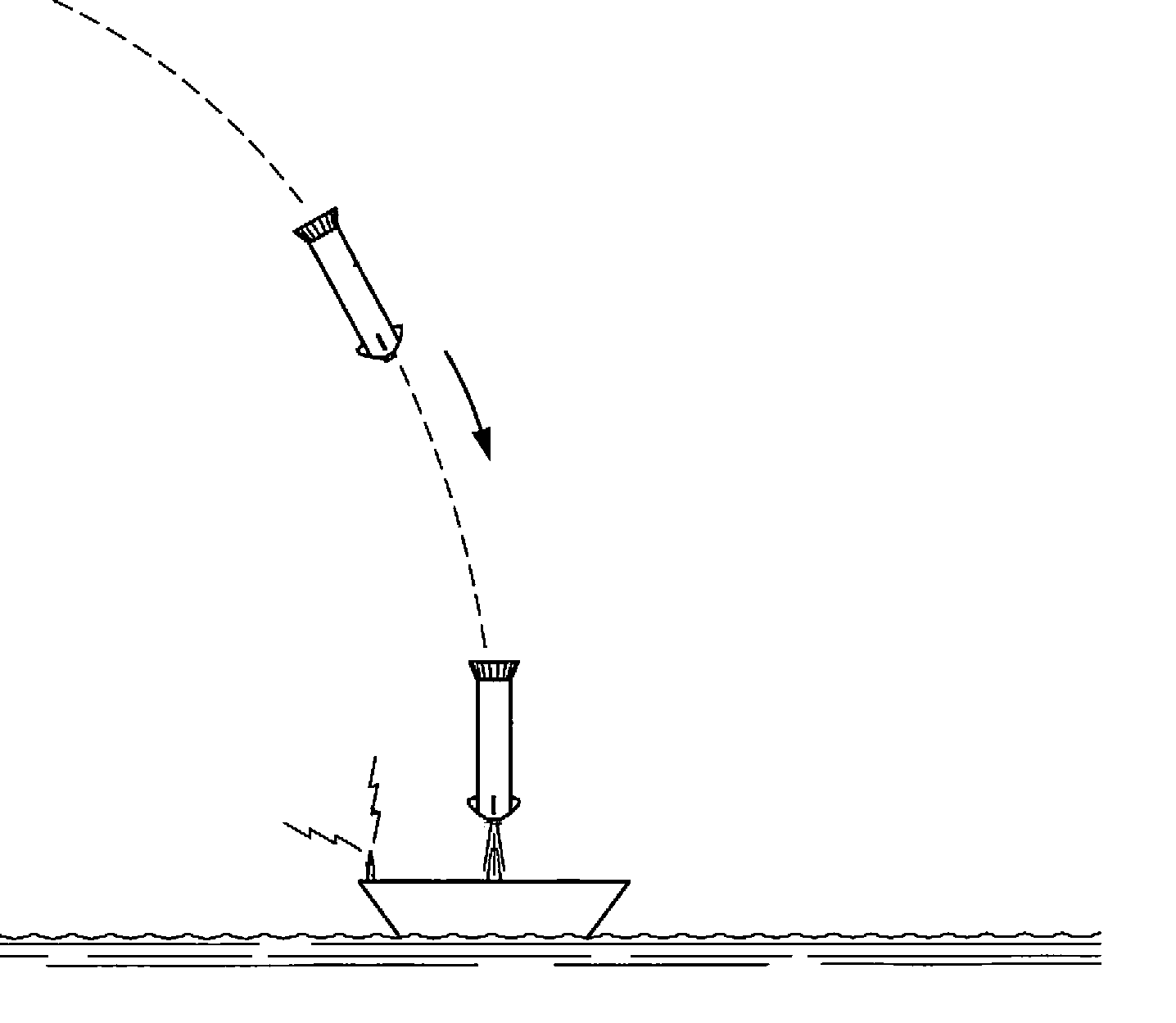}
\vspace{-0.25cm}
\caption{Vehicle reentry onto a sea landing platform. \label{fig:spaceX}}
\end{center}
\end{figure}

For the two variants of the optimal state and input estimator proposed in this paper, we assume: 

\noindent {($A1$)} \emph{ELISE}:
The process noise $w$ and the measurement noise $v$ are assumed to be mutually uncorrelated, zero-mean, white random signals with known covariance matrices, with noise statistics $Q={\rm diag}(5\times 10^{-4},  10^{-4})$ and $R={\rm diag}(10^{-5}, 10^{-4},10^{-5})$.  An additional measurement of range acceleration $\overline{y}(t)=\dot{h}_1(t)+\overline{v}(t)$, is available\footnote{Although range acceleration measurement may be accessible with the use of an accelerometer, we used the \emph{filtered} derivative of $y_3$, i.e., $\overline{y}(s)=\frac{s}{0.05s+1} y_3(s)$ ($s$ is the Laplace variable), as the additional measurement to illustrate the possibility of using such an approach with ELISE.} with $\overline{R}=0.75$ and $\grave{R}=[0.0866, 0, 0.0274]^\top$. 

\noindent {($A1'$)}  \emph{ALISE}: The noise signals 
are Gauss-Markov processes:
$\dot{w}+0.2 I_3 w = w_G,\
\ddot{v}+ I_3 \dot{v} + 0.25 I_3 v =v_G$,
where $w_G$ and $v_G$ are mutually uncorrelated, zero-mean, white noise signals with intensities $Q_G ={\rm diag}(5 \times 10^{-4}, 10^{-4}, 2 \times 10^{-4})$ and $R_G ={\rm diag}(5 \times 10^{-3}, 10^{-4}, 10^{-5})$, respectively.

\begin{figure}[!tp]
\begin{center}
\includegraphics[scale=0.35,trim=20mm 5mm 5mm 7mm,clip]{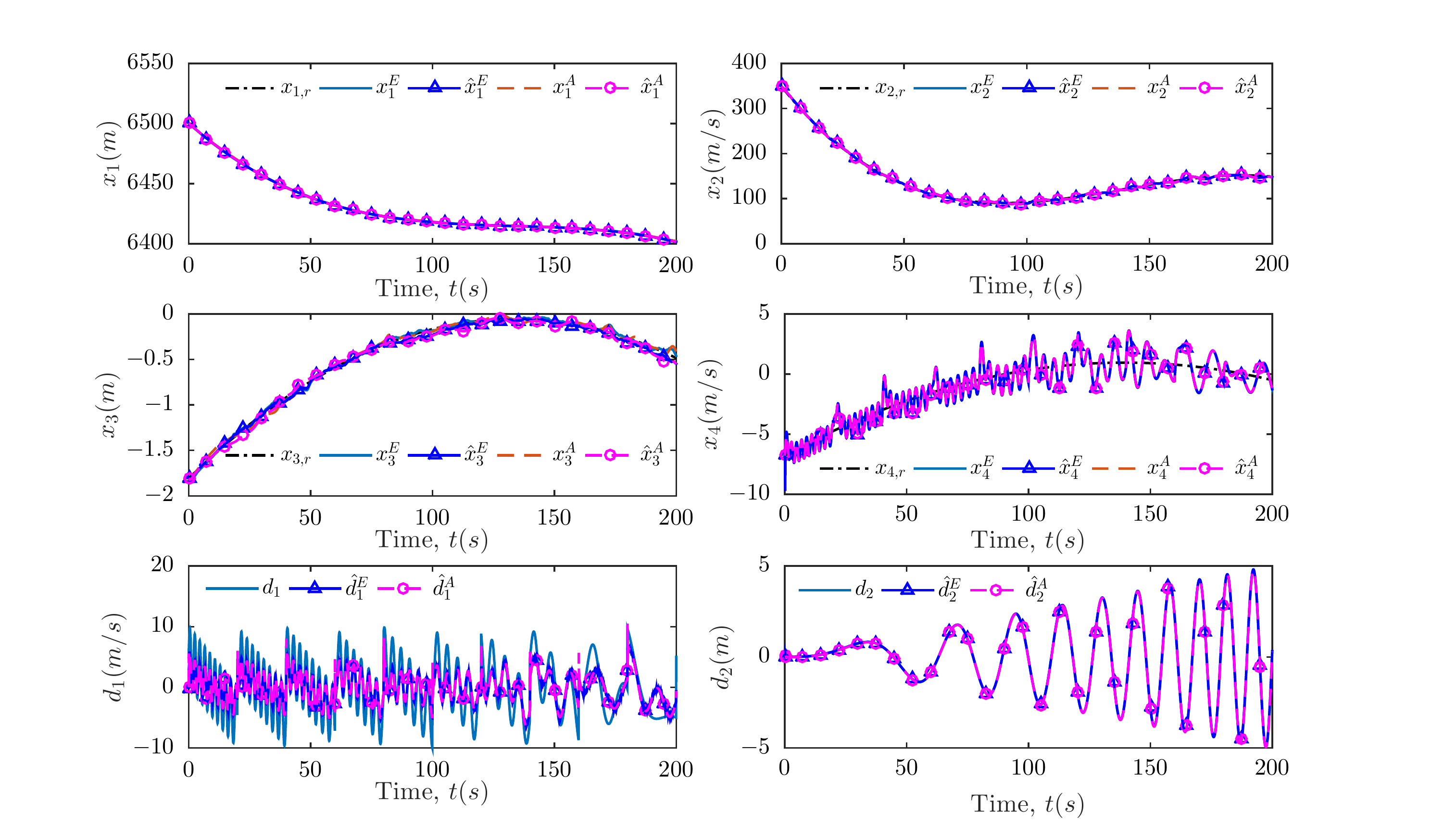}\vspace*{-0.5cm}
\caption{Actual states $x_1$, $\cdots$, $x_4$ and its estimates $\hat{x}_1$, $\cdots$, $\hat{x}_4$; unknown inputs $d^w_1$, $d^e_2$, and its estimates $\hat{d}_1$, $\hat{d}_2$; Superscripts $E$ and $A$ denote ELISE and ALISE, respectively (averaged over 100 simulations). \label{fig:v1}}
\end{center}
\end{figure}

\begin{figure}[!tp]
\begin{center}
\includegraphics[scale=0.355,trim=22.5mm 2mm 5mm 10mm,clip]{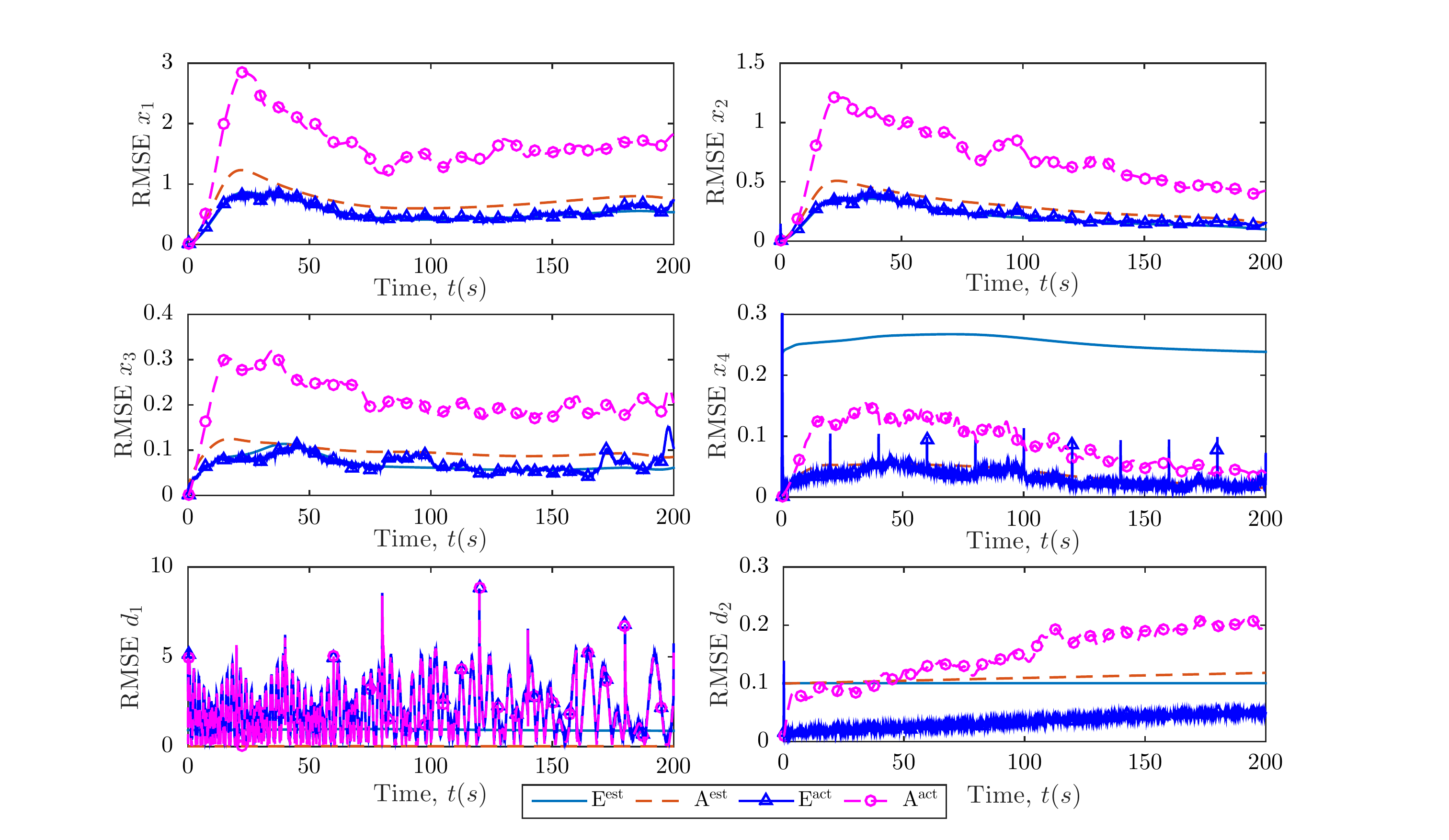}\vspace*{-0.5cm}
\caption{Root mean squared errors (RMSE) of state $x_1$ through $x_4$ and unknown input estimates $d^w_1$ and $d^e_2$ computed from 100 simulations; $\mathrm{E}$ and $\mathrm{A}$ denote ELISE and ALISE, respectively;  Superscripts $\mathrm{est}$ and $\mathrm{meas}$ denote estimated (i.e., from $P^x$) and measured/actual RMSE values. \label{fig:v2}}
\end{center}
\end{figure}

Since we have a separation principle for the controller and estimator (Theorem \ref{thm:separation}), we can design them independently. The controller for this example is chosen as
\begin{align*}
u_1&=u_{1,r}-\hat{\mathcal{D}}\hat{x}_3-\hat{\mathcal{G}}\hat{x}_1-k_D(\hat{x}_3-x_{3,r})-k_P(\hat{x}_1-x_{1,r}),\\ 
u_2&=u_{2,r}-\hat{\mathcal{D}}\hat{x}_4-\hat{\mathcal{G}}\hat{x}_2-k_D(\hat{x}_4-x_{4,r})-k_P(\hat{x}_2-x_{2,r}),
\end{align*}
where $u_{1,r}$ and $u_{2,r}$ are the reference inputs corresponding to the reference trajectory, $\hat{\mathcal{D}}$ and $\hat{\mathcal{G}}$ are estimates of $\mathcal{D}$ and $\mathcal{G}$, while $k_D=1.8$ and $k_P=1$ are controller gains (chosen via pole placement at $-0.9000 \pm 0.4359i$).
Note that the system \eqref{eq:example} in this example becomes unstable when only the reference input is applied; thus, the stabilizing controller above is necessary. 

For disturbance rejection, we chose $J_1=[0,1]^\top$ and
$J_2=[0,0]^\top$, since we observe that the \emph{matching condition} (cf. Section \ref{sec:sepPrinciple}) holds.  
For the ALISE variant, 
$\mathfrak{d}t$ is chosen as $0.05s$. We implemented the above state feedback control law and both filter variants described above in MATLAB/Simulink on a 2.2 GHz Intel Core i7 CPU, with initial states $x(0)=[6500.4,349.14,-1.8093,-6.7967,0.6932]^\top$ and non-periodic and non-smooth unknown inputs depicted in Fig. \ref{fig:v1} (e.g., $d^w_1$ is composed of sawtooth and chirp signals).

 Fig. 
\ref{fig:v1} shows the actual and estimated system states $x_1$ through $x_4$, as well as unknown inputs $d^w_1$ and $d^e_2$, averaged over 100 Monte Carlo simulations. We observe that both proposed filters, ELISE and ALISE, estimate these system states and unknown inputs reasonably well. On the other hand, we see from Fig. \ref{fig:v2} that the estimated root mean squared errors (RMSE) are, with the exceptions of $x_4$ and $d^e_2$, higher than the actual/measured RMSE values. The RMSE of ALISE also appears higher than that of ELISE. 
These discrepancies may be due to approximations associated with the use of linearized dynamics. Note that the state $x_5$ (not depicted due to space constraints), which we recall to be the unknown aerodynamic parameter, is not as well estimated with our filters. However, this is not a problem, as the main objective of the vehicle reentry problem is the tracking of the reference trajectory, which is demonstrated to be successful with our filters. 

Moreover, it is noteworthy that ALISE performs reasonably well, despite the fact that $\alpha_2$ in \eqref{eq:alpha23} is unbounded because of the unboundedness of $\dot{d}^w_1$ (due to its sawtooth component). This suggests that the supremum in $\alpha_2$ may be taken over the set with nonzero measure only.

\subsection{Hover Control of a Helicopter} \label{sec:hover}

\begin{figure}[!htp]
\begin{center}
\includegraphics[scale=0.2]{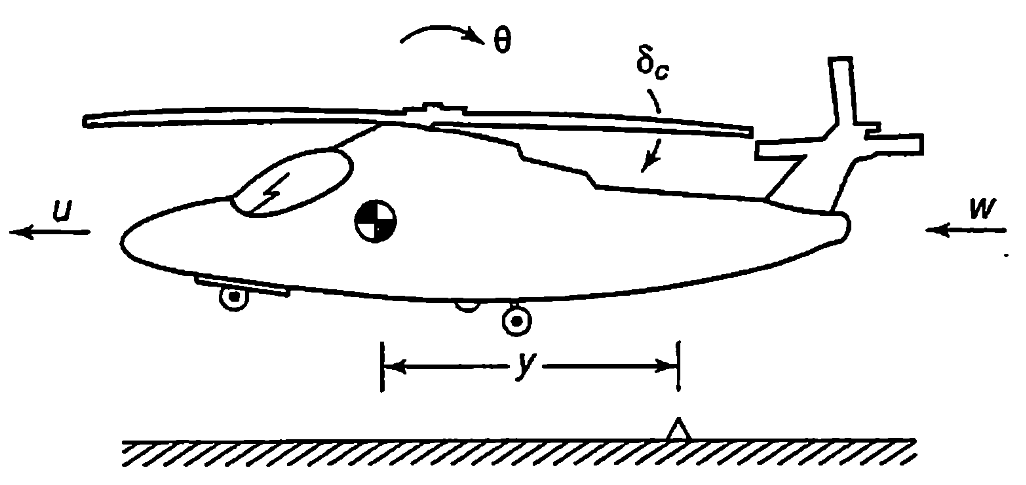}
\vspace{-0.25cm}
\caption{A helicopter near hover \cite{bryson.2002}. \label{fig:helicopter}}
\end{center}
\end{figure}

Next, we consider an example with a helicopter depicted in Fig. \ref{fig:helicopter} with the following longitudinal dynamics \cite{bryson.2002}:
\begin{align} \label{eq:example}
\begin{array}{ll}
\dot{\theta}&= q,\
\dot{q}=-0.415 q-0.011 u +6.27 \delta_c - 0.011 w_h,\\
\dot{y}&=u, \ \dot{u} =9.8 \theta -1.43 q -0.0198 u +9.8 \delta_c -0.0198 w_h,
\end{array}
\end{align} \noindent
where the system states, $\mathbf{x}:=\begin{bmatrix} \theta & q & u & y\end{bmatrix}^\top$, are the fuselage pitch angle $\theta$, the pitch rate $\dot{q}$, the horizontal velocity of the center of gravity $u$ and the horizontal distance from the desired hover point $y$; while the only control input is the tilt angle of the rotor thrust vector $\delta_c$. The variable $w_h=w_d+w$ represents a horizontal wind disturbance, which consists of a deterministic time-varying component $w_d$ and  a stochastic component $w$. We have noisy measurements of $y$, $u$ and $q$ only, with a time-varying output matrix given by
{\small$ C= \begin{bmatrix} 0 & 0 & 0 & 1 \\ 0 & 0 & 0.8 + 0.2 \sin t & 0 \\ 0 & 1 & 0 & 0 \end{bmatrix}.$}
Moreover, the measurement of $u$ is plagued by an additive time-varying bias, $e_m$. Thus, the measurement vector is given by $z=\begin{bmatrix} y\ & (0.8 + 0.2 \sin t) u + e_m  & \ q\end{bmatrix}^\top$. In this example, $w_d$ and $e_m$ are sawtooth and sinusoidal signals, respectively.

For the two variants of the optimal state and input estimator proposed in this paper, we assume: 

\noindent {($A1$)} \emph{ELISE}:
The process noise $w$ and the measurement noise $v$ are assumed to be mutually uncorrelated, zero-mean, white random signals with known covariance matrices, with noise statistics $Q=5 \times 10^{-4}$ and $R={\rm diag}(1 \times 10^{-3}, 1.6 \times 10^{-3},0.9 \times 10^{-3})$.  An additional measurement of linear acceleration (e.g., from an accelerometer),  $\overline{y}=\dot{u}$, is available with $\overline{R}=2 \times 10^{-3}$.

\noindent {($A1'$)}  \emph{ALISE}: The process noise, $w$ is a first-order Gauss-Markov process, and the measurement noise, $v$, is a second-order Gauss-Markov process:
$\dot{w}+0.2 w = 6 w_G,\
\ddot{v}+ \dot{v} + 0.25 v =v_G$,
where $w_G$ and $v_G$ are mutually uncorrelated, zero-mean, white noise signals with intensities $Q_G =5 \times 10^{-4}$ and $R_G ={\rm diag}(1 \times 10^{-3}, 1.6 \times 10^{-3},0.9 \times 10^{-3})$, respectively.

Note that the system \eqref{eq:example} in this example is open-loop unstable; thus, a stabilizing controller is necessary. Since we have a separation principle for the controller and estimator (Section \ref{sec:sepPrinciple}), we can design them independently. The controller we chose is the LQR, and the estimator is the filter proposed in this paper. For the LQR, we have chosen the following cost matrices: $Q_{LQR}=C_{LQR} ^\top C_{LQR}$ and $R_{LQR}=5$, where $C_{LQR}:=\begin{bmatrix} 0 & 0 & 0 & 1 \end{bmatrix}$. For minimizing the effect of disturbance input on the closed loop system, we solve the semidefinite programs described in Section \ref{sec:sepPrinciple} using an off-the-shelf software package CVX \cite{cvx,Grant.08} 
to obtain $J_1=0$, $J_2=-1.943 \times 10^{-3}$. With the ALISE variant, $J_2$ is chosen to be zero, such that the error induced by finite difference approximations is not amplified in \eqref{eq:dexp},
whereas $\mathfrak{d}t$ is chosen as $0.05~s$.

\begin{figure}[!htp]
\begin{center}
\subfigure[With the ELISE variant]{
\includegraphics[scale=0.32,trim=26.75mm 9mm 44.25mm 20mm,clip]{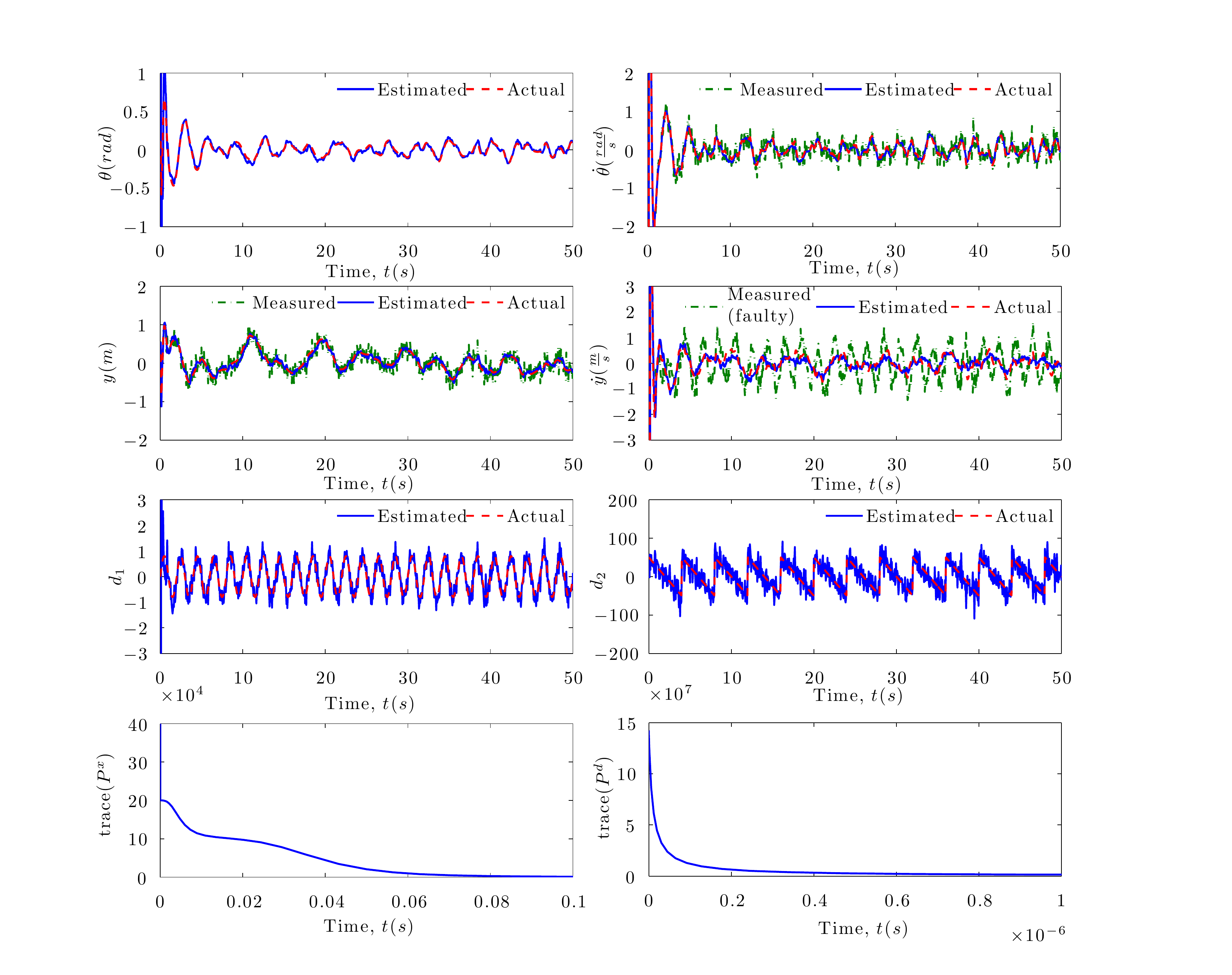}}
\subfigure[With the ALISE variant]{
\includegraphics[scale=0.32,trim=26mm 9mm 43mm 15mm,clip]{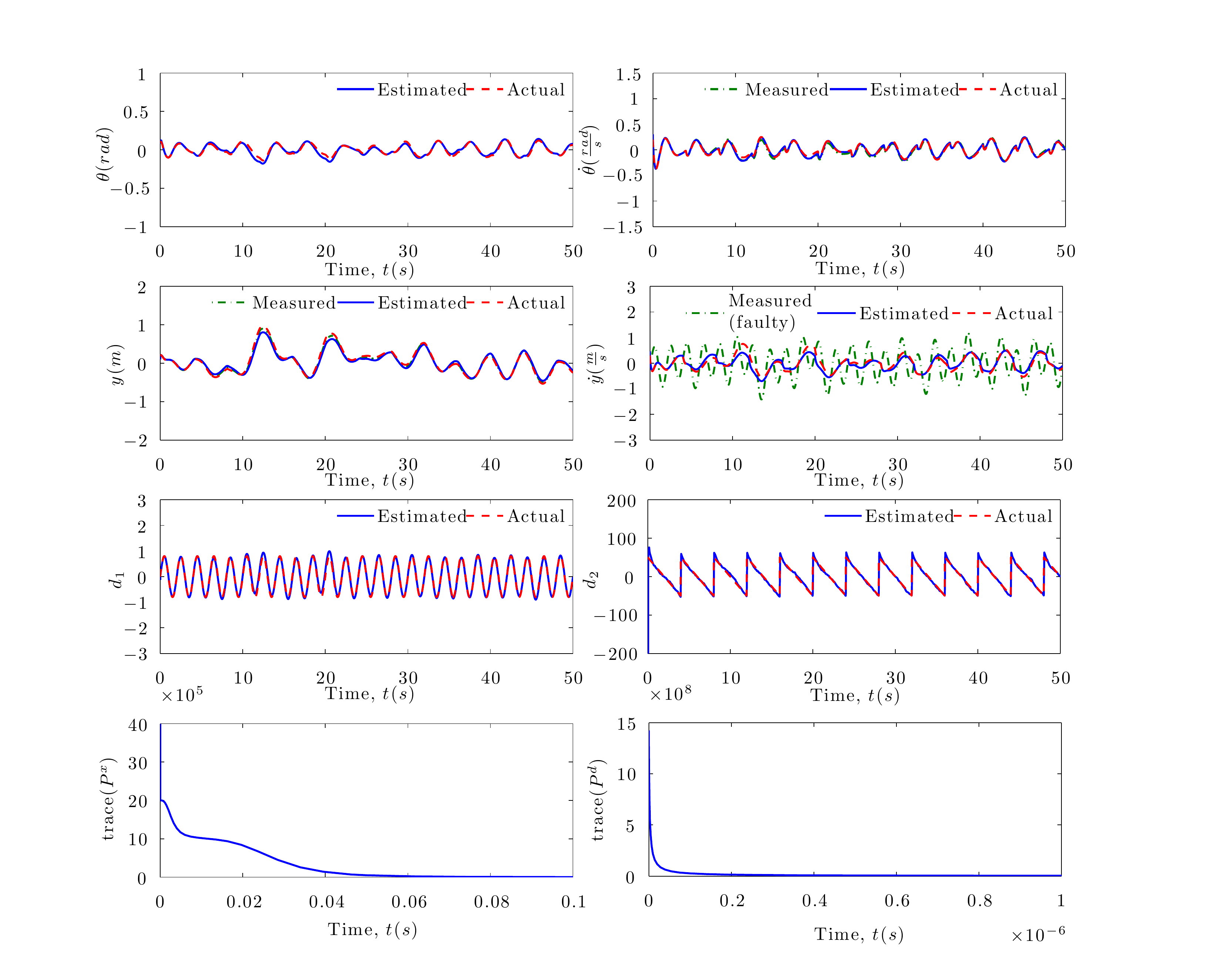}}
\caption{Actual states $\theta$, $q$, $y$, $u$ and its estimates; unknown inputs $d_1$, $d_2$, and its estimates; as well as the trace of continuous estimate error covariance of states and unknown inputs. \label{fig:v2}}
\end{center}
\end{figure}

We implemented the LQR state feedback control law and both filter variants described above in MATLAB/Simulink on a 2.2 GHz Intel Core i7 CPU. Fig. 
\ref{fig:v2} shows the actual and estimated system states, as well as unknown inputs. Note that the projections of the unknown input vector, i.e. $d_1$ and $d_2$, obtained with the transformation \eqref{eq:T_k}, correspond to real unknown signals, in that $d_1=e_m$ and $d_2=w_d$. Thus, we observe from the figures that the proposed filter successfully estimates the system states and also the unknown inputs, $w_d$ and $e_m$, and the traces of the continuous estimate error covariance matrices of both states and unknown inputs converge in less than $0.1~ s$.
However, the convergence rate of the trace of estimate error covariance matrices of ELISE is slower than that of ALISE, while the actual states of the system appear less noisy in ALISE. The reason behind these is the difference in assumed noise models in both variants. For ALISE, the process noise is a `filtered' white noise, whereas for ELISE, the process noise is `unfiltered' and there are two sources of measurement noises, through the output and output derivatives. As before, ALISE performs reasonably well, despite the fact that $\alpha_2$ in \eqref{eq:alpha23} is unbounded given an unbounded $\dot{d}_2$ signal on a set of measure zero. This suggests that the supremum in $\alpha_2$ can only be taken over the set with nonzero measure.
Besides, we observe from the simulations that the finite difference approximation in the ALISE algorithm functions as a low-pass filter of sorts for the input estimate. If $\mathfrak{d}t$ is small, the input estimate appears noisy, whereas a large value of $\mathfrak{d}t$ `smooths' out the high frequencies in the unknown input estimate.

\section{Conclusion} \label{sec:conclusion}
This paper presented an optimal filter for linear time-varying continuous-time stochastic systems that simultaneously estimates the states and unknown inputs in an unbiased minimum-variance sense. We showed that the unknown inputs cannot be estimated without additional assumptions and discussed two variants of the filter: one with an `output derivative' measurement and another without such a measurement. 
The properties of our filter are derived by constructing a `virtual' equivalent system without unknown inputs, which has analogous properties to the Kalman-Bucy filter. Moreover, using limiting case approximations, we find that the optimal discrete-time filter implicitly uses finite difference to obtain an `output derivative'. We also presented conditions under which the proposed filter is uniformly asymptotically stable, and has a steady-state solution, as well as provided the convergence rate of the filter estimates. 
In addition, we showed that a principle of separation of estimation and control also holds for linear systems with unknown inputs and that disturbance rejection is possible. Simulation examples of a nonlinear vehicle reentry problem and a helicopter hover control problem demonstrate the claims in this paper.


\section*{Acknowledgments}
This work  was supported in part by the National Science Foundation, grant \#1239182. M. Zhu was partially supported by ARO W911NF-13-1-0421 (MURI), NSA H98230-15-1-0289 and NSF CNS-1505664.

\bibliographystyle{unsrt}
\bibliography{biblio}

\appendix

In this Appendix, we first provide proofs of Lemmas \ref{lem:convergence}, \ref{lem:convergenceInput}, \ref{lem:convergence2} and \ref{lem:convergenceInput2} on the convergence of state and input estimate biases of ELISE and ALISE. 
Then, we prove the claim of optimality of ELISE in the minimum-variance unbiased sense in Theorem \ref{thm:main} by first constructing a `virtual' equivalent system without unknown inputs with analogous properties to the Kalman-Bucy filter. Then, we provide an alternative derivation by means of limiting case approximations of the optimal discrete-time filter presented in a previous work \cite{Yong.Zhu.ea.Automatica16}, which we observe to implicitly use finite difference to obtain an `output derivative'. Next, we show in Appendix \ref{sec:equivAlg} that \eqref{eq:xhat} and \eqref{eq:xhat2} are equivalent, from which it follows that the state estimate of ALISE is optimal (Theorem \ref{thm:main2}). 
We then derive the conditions under which the optimal filter is uniformly asymptotically stable, given in Theorems \ref{thm:stable} and \ref{thm:stable2}.
Finally, we provide the convergence proofs of Theorems \ref{thm:conv} and \ref{thm:conv2} as well as a proof of Proposition \ref{remark}.

\subsection{Proof of Lemmas \ref{lem:convergence} and \ref{lem:convergence2}}  \label{sec:stateConv}

In this section, we derive the convergence rate of the state estimate bias that was given in Lemmas \ref{lem:convergence} and \ref{lem:convergence2}.
We first provide a proof for the convergence rate of the expected state estimate bias for ELISE. Then, we show that the same convergence rate is true of ALISE, by showing that the state estimate of ELISE and ALISE are equivalent.

\subsubsection{Error Bound on State Estimate for ELISE}
From \eqref{eq:dhat} and choosing the matrices $M_1$ and $M_2$ such that $M_1 \Sigma =I$ and $M_2 \overline{C}_2 G_2=I$, which is possible because $\Sigma$ and $\overline{C}_2 G_2$ have full rank by assumption, we obtain
\begin{align} \label{eq:dtilde}
\begin{array}{ll}
\tilde{d}_1:=d_1-\hat{d}_1&=-M_1 (C_1 \tilde{x}+ v_1),\\
\tilde{d}_2:=d_2-\hat{d}_2&=-M_2 (\overline{C}_2 \hat{A}+\overline{T}_2 \doverline{C}) \tilde{x} - M_2 \overline{v}_2 \\ &\qquad + M_2 \overline{C}_2 G_1 M_1 v_1-M_2 \overline{C}_2 W w,
\end{array}
\end{align} \noindent
where $\hat{A}:=A-G_1 M_1 C_1$. Note that $\overline{v}_2=T_2 \dot{v}$ in the case when the signal $\dot{y}$ is known, as is assumed for ALISE. Next, substituting \eqref{eq:dtilde} into the system dynamics in \eqref{eq:sys}, and using \eqref{eq:xhat}, we obtain the state estimate error system
\begin{align} \label{eq:xtilde}
\dot{\tilde{x}}  &= \dot{x}- \dot{\hat{x}}=\overline{A} \tilde{x} + \overline{w} - L(C_2 \tilde{x} + v_2),
\end{align} \noindent
where $\overline{A}$ and $\overline{w}$ are as defined in Theorem \ref{thm:stable}. 

The state estimate bias system, $\dot{\breve{x}}=\breve{A} \breve{x}$ (from \eqref{eq:xtilde}), is linear, where $\breve{A}:=\overline{A}-L C_2$ and $\breve{x}:=\mathbb{E}[\tilde{x}]$. Since we assume that the filter is uniformly asymptotically stable and the state estimate bias system is linear, by \cite[Theorem 3.3.8]{Hinrichsen.2005} and \cite[Theorem 3]{Kalman.Bertram.1960}, the resulting state estimate bias of the system 
decays exponentially, i.e., 
 there exist $\gamma$ and $\beta$ such that the state estimate bias converges exponentially as is given in \eqref{eq:stateError}. 

If additionally, $\breve{A}$ is bounded, then by the result in \cite[Theorem 5]{Silverman.1968},
the pair $(\breve{A},I)$ is uniformly completely observable (see Definition \ref{def:uco}), where $I_{n \times n}$ is the identity matrix.
 Next, since $\breve{A}$ is Hurwitz and 
 $(\breve{A},I)$ is uniformly completely observable, with bounded $I$ and
$\breve{A}$, we can apply the result of \cite[Theorem (i)]{Anderson.1967} and \cite[Theorem 5(i)]{Anderson.1969} to obtain explicit expressions of the constants $\beta$ and $\gamma$. 
Besides, given that $(\breve{A},I)$ is uniformly completely observable, 
then there exists a unique positive definite solution, 
$S(t)=\lim_{T \to \infty} \Pi(t,T) \succ 0$ for all $t  \geq t_0$, 
where $\Pi(t,T)$ is defined by $\dot{\Pi}+(\overline{A}-L C_2)^\top \Pi + \Pi (\overline{A}-L C_2)=-I $ with boundary condition $\Pi(T,T)=0$. In addition, 
$S(t)= \displaystyle \lim_{T \to \infty} \int^{T}_{t} \Phi_{\breve{A}(t)}(t,s) \Phi_{\breve{A}(t)}^\top(t,s) ds$, which has eigenvalues that are bounded above and below \cite[Eq. (20-21)]{Anderson.1967} and $V=\breve{x}^\top S \breve{x}$ is a Lyapunov function with $\dot{V}=-\breve{x}^\top  \breve{x}$ \cite[Eq. (23)]{Anderson.1967}. 
For the detailed proof of this, the reader is referred to \cite{Anderson.1967,Anderson.1969}.

From the Lyapunov function above, we apply the approach in \cite[pp. 91-93]{slotine.li.book91} to analyze the convergence rate of the state estimate bias.
Let $\lambda_{max}(S)$ denote the largest eigenvalue of $S$ 
and $\overline{\gamma}:=1/\lambda_{max}(S)$. Then, from
\begin{align*}
\breve{x}^\top \breve{x} \geq  \frac{1}{\lambda_{max}(S)} \breve{x}^\top (\lambda_{max}(S) I) \breve{x} \geq \overline{\gamma} \breve{x}^\top S \breve{x} = \overline{\gamma} V,
\end{align*}
\noindent we have $\dot{V} \leq - \overline{\gamma} V \leq - 2 \gamma V$, where $\gamma$ is the supremum of $\overline{\gamma}/2$ over the set of all $\overline{\gamma}$ for all $t\geq t_0$. Then, applying the convergence lemma in \cite[p. 91]{slotine.li.book91}, we have $$\underline{\lambda}_{min}(S) \| \breve{x}\|^2 \leq
\lambda_{min}(S) \| \breve{x}\|^2 \leq \breve{x}^\top S \breve{x} \leq V(t_0) e^{-2{\gamma}t}.$$ 
Hence, we obtain $\| \mathbb{E}[\tilde{x}] \| = \| \breve{x} \| \leq \beta e^{-\gamma (t-t_0)}$ where $\beta$ and $\gamma$ are given in \eqref{eq:betagamma}. 

\subsubsection{Error Bound on State Estimate for ALISE}  \label{sec:equivAlg}

In this section, we provide an alternative to ELISE for estimating the state of the system of interest, when an additional measurement $\overline{y}$ containing information that is equivalent to the `output derivative' is unavailable, which is a central feature of ALISE. First, we note that the additional measurement would be superfluous if the output derivative is fortuitously available. Thus, the idea is to derive a state estimator through indirect access of $\dot{y}$ with only measurements of $y$. This same idea would also apply for cases when $\dot{u}$ is not easily computed. Therefore, to circumvent the need to have direct access to $\dot{y}$ and $\dot{u}$ in ALISE, we propose an equivalent state estimation algorithm given by \eqref{eq:xhat2} that produces the same state estimate as \eqref{eq:xhat} with only $y$ and $u$, which are known.
Using \eqref{eq:dhat1} and \eqref{eq:variant1} with $\overline{z}_2=T_2 \dot{y}$ and the matrices according to Special Case \ref{sp} as well as rearranging and combing terms, the state estimation \eqref{eq:xhat} can be rewritten as follows:
\begin{align}
\begin{array}{ll}
 \dot{\hat{x}}&=A \hat{x} + B u + G_1 M_1(z_1-C_1 \hat{x}-D_1 u)+G_2 M_2\\
 &\quad (  T_2 \dot{y} -(C_2 A+T_2 \dot{C}) \hat{x}-C_2 G_1 M_1(z_1-C_1 \hat{x}-D_1 u)\\
  &\quad -D_2 \dot{u}-(C_2 B +T_2 \dot{D}) u) + L (z_2-C_2 \hat{x} -D_2 u)  \\ 
 &=(\overline{A}-L C_2) \hat{x} + (\overline{B}-LD_2)u + \overline{G} M_1 z_1 + L z_2 \\
 &\quad+G_2M_2 T_2 \dot{y}-G_2M_2 D_2 \dot{u}\\
 &:=g(\hat{x},u,z_1,z_2)+\Phi_1 \dot{y} + \Phi_2 \dot{u}, \vspace{-0.45cm}
 \end{array} \label{e2}
 \end{align} \noindent
 where $\overline{B}:=(I-G_2M_2 C_2)(B-G_1M_1D_1)$, $\overline{G}=(I-G_2 M_2 C_2) G_1 M_1$.
Then, to derive an equivalent without $\dot{y}$ and $\dot{u}$, we let
\begin{align}
\begin{array}{rl}\dot{\theta} &= g(\check{x},u,z_1,z_2)-\dot{\Phi}_1 y - \dot{\Phi}_2 u, \\
{\check{x}}& = \Phi_1 y + \Phi_2 u + \theta,\end{array}\label{e1}
\end{align} \noindent
where $\dot{\Phi}_1$ and $\dot{\Phi}_2$ can be obtained by differentiating ${\Phi}_1$ and ${\Phi}_2$. The resulting equations are summarized in Algorithm \ref{algorithm2}.
Taking the derivative of ${\check{x}}$, we have
\begin{align*} \dot{{\check{x}}} &= \Phi_1 \dot{y} +\dot{\Phi}_1 y + \dot{\Phi}_2 u+ \Phi_2 \dot{u} + \dot{\theta} \\&= \Phi_1\dot{y} + \Phi_2 \dot{u} + g({\check{x}},u,z_1,z_2).
\end{align*}
So the output $\check{x}$ of~\eqref{e1} is identical to that of $\hat{x}$ in~\eqref{e2}. However,~\eqref{e1} does not include $\dot{y}$ and $\dot{u}$, as desired. Nevertheless, because of the different assumed noise models, the resulting filter gain $L$ is different in \eqref{eq:xhat} and \eqref{eq:xhat2}. The filter gain equation and Riccati differential equation remain the same, 
but are computed with different noise covariance matrices of $w$, $v$ and $\overline{v}:=\dot{v}$: 
\begin{align*}
Q &= P^w,&  
R =\begin{bmatrix} I & 0 \end{bmatrix} P^v \begin{bmatrix} I & 0 \end{bmatrix}^\top,\\
\overline{R}&=\begin{bmatrix} 0 & I \end{bmatrix} P^v \begin{bmatrix} 0 & I \end{bmatrix}^\top,&
\grave{R}=\begin{bmatrix} I & 0 \end{bmatrix} P^v \begin{bmatrix} 0 & I \end{bmatrix}^\top,
\end{align*}
where $P^w$ and $P^v$ are propagated covariances that are solutions to the differential Lyapunov equations given by $\dot{P}^w=-{A}_w P^w -P^w \overline{A}_w^\top + B_w Q_G B_w^\top$ and $\dot{P}^v=\overline{A}_v P^v +P^v \overline{A}_v^\top + \overline{B}_v R_G \overline{B}_v^\top$, with $P^w(t_0)=\mathcal{P}^w_0$ and $P^v(t_0)=\mathcal{P}^v_0$, respectively, as given in \cite{Simon.2006}. Note that $P^w$ and $P^v$ are bounded for all $t \geq t_0$, and their bounds are given in \cite{Hmamed.1990}. 

\subsection{Proof of Lemmas \ref{lem:convergenceInput} and \ref{lem:convergenceInput2}}
We now prove Lemmas \ref{lem:convergenceInput} and \ref{lem:convergenceInput2}, which give a bound on the input estimate bias as a function of time, $t$, and time difference of the finite difference approximation, $\mathfrak{d}t$, which results from a biased initial state estimate, and is induced by the finite difference approximation in ALISE.

\subsubsection{Error Bound on Input Estimate for ELISE}
Since we have shown in Appendix \ref{sec:stateConv} that $\|\mathbb{E}[\tilde{x}]\|$ converges to zero with rate $\gamma$, it follows from \eqref{eq:dtilde} that
\begin{align*}
\begin{array}{rl}
 \|\mathbb{E}[\tilde{d}]\| &\leq  \| V_1 \mathbb{E}[\tilde{d}_1] \|+ \| V_2\mathbb{E}[\tilde{d}_2]\|\\
&\leq (\|V_1 M_1 C_1 \| +\|V_2 M_2 (\overline{C}_2\hat{A}+\overline{T}_2\doverline{C})\|) \beta e^{-\gamma (t-t_0)},\end{array} 
\end{align*} \noindent
i.e., the convergence rate of input estimate bias is also $\gamma$.

\subsubsection{Error Bound on Input Estimate for ALISE \label{sec:errorBound}}

Unlike the state estimate, the estimate of the unknown input can only be computed to any degree of accuracy when compared to the MVU input estimate assuming that the exact output derivative is known. This is not unexpected, as this is also the same extent that observer designs (e.g.,\cite{Corless.98,Xiong.2003}) are able to achieve. Thus, in this section, we provide the expected error bound on the unknown input estimate given by \eqref{eq:variant2}, which, asymptotically, is arbitrarily small.

As seen in \eqref{eq:variant2}, the ALISE algorithm utilizes the backward finite approximation of the output derivative. This induces an error in the estimate $\hat{d}_2$, when compared to the ideal case in which $\dot{z}_2=T_2 \dot{y}$ ($\dot{T}_2=\dot{U}_2^\top=0$ by Theorem \ref{thm:Hdot}) is accessible. The next lemma characterizes the effect of the approximation error on the estimate, specifically, on the bias and variance of the estimate, $\mathbb{E}[d_2-\hat{d}_2]$ and $\mathbb{E}[(d_2-\hat{d}_2)(d_2-\hat{d}_2)^\top]$.

\begin{lem} \label{lem:Taylor}
The error induced in the estimate of $\hat{d}_2$ by replacing the exact $\dot{z}_2=T_2 \dot{y}$ with its finite difference approximation is given by
\begin{align}
  {\overline{d}}_2- \hat{d}_2&= M_2 \mathbb{E}[ \dot{z}_2-\frac{z_2(t)-z_2(t- \mathfrak{d}t)}{\mathfrak{d}t}] = \frac{1}{2}  M_2 \mathbb{E}[\ddot{z}_2(c)] \mathfrak{d}t \label{eq:errorApprox}
\end{align} \noindent
for some $c \in (t-\mathfrak{d}t, t)$, where the input estimate with perfect knowledge of $\dot{z}_2$ is defined as $\overline{d}_2:=M_2(T_2 \dot{y}-(C_2A + T_2 \dot{C} \hat{x}-C_2 B u - C_2 G_1 \hat{d}_1 -D_2 \dot{u}-T_2 \dot{D} u))$ and
\begin{align*}
\begin{array}{ll}
\ddot{z}_2 &=T_2[(2 \dot{C} A +C \dot{A}+C A^2 + \ddot{C})x+(\dot{C} B + C A B + C \dot{B} \\
&\quad + \dot{C} B + \ddot{D})u + (CB + 2 \dot{D}) \dot{u} + D \ddot{u} +(2 C \dot{G}_1 + CA G_1  \\
&\quad + C \dot{G}_2) d_1 +C G_2 \dot{d}_1 +\dot{H}_1 \dot{d}_1 + \ddot{H}_1 d_1 + H_1 \ddot{d}_1 + \dot{H}_1 \dot{d}_1\\
&\quad +(2 \dot{C} G_2 +CAG_2 + C \dot{G}_2) d_2 + C G_2 \dot{d}_2+(\dot{C}W + C \dot{W}  \\
&\quad - C WA_w)w+C W B_w w_G + \ddot{v}]\\
&=\mathbb{E}[\ddot{z}_2]+T_2 ((\dot{C}W + C \dot{W} - C WA_w)w+ C W B_w w_G  \\
&\quad -A_{\dot{v}} \dot{v} - A_v v+ B_v v_G).
\end{array}
\end{align*}
\end{lem}
\begin{proof}
To obtain the above result, we apply Taylor's theorem (see for e.g., \cite{Rudin.1976} for proof of Taylor's theorem) to obtain
\begin{align} \label{eq:taylor}
z_2(t-\mathfrak{d}t)=z_2(t)-\dot{z}_2((t)) \mathfrak{d}t + \frac{1}{2} \ddot{z}_2(c) (\mathfrak{d}t)^2
\end{align} \noindent
for some $c \in (t-\mathfrak{d}t, t)$, since by Assumption ($A1'$), $\dot{z}_2$ is continuous on $[t-\mathfrak{d}t, t]$ and $\ddot{z}_2$ exists for all $t \in [t-\mathfrak{d}t, t]$. Rearranging the above equation, we have
\begin{align*}
\nonumber \dot{z}_2-\frac{z_2(t)-z_2(t- \mathfrak{d}t)}{\mathfrak{d}t}= \frac{1}{2} \ddot{z}_2(c) \mathfrak{d}t.
\end{align*}
Then, by differentiating $z_2$ twice with respect to $t$ and from \eqref{eq:variant2}, we find the error induced by the finite difference approximation as given in \eqref{eq:errorApprox}.
\end{proof}
\begin{rem}
Note that $\ddot{z}_2$ and $\mathbb{E}[\ddot{z}_2]$ are independent of $\dot{H}_1$, $\ddot{H}_1$ and $\ddot{d}_2$ by Corollary \ref{cor:cor1}. Thus, for the boundedness of $\ddot{z}_2$, the signal $\ddot{d}$ need not be bounded, similar to the assumption for the observer in \cite{Corless.98}. Furthermore, the derivatives of $H$ need not be bounded.
\end{rem}

Armed with Lemma \ref{lem:Taylor}, we now derive the expected error bound in Lemma \ref{lem:convergenceInput}.
The total expected input estimate bias consists of the error given in \eqref{eq:derror} due to initial state estimate bias, and the error induced by the finite difference approximation $\dot{z}_2$ given by Lemma \ref{lem:Taylor}, i.e.,
\begin{align}
\begin{array}{ll}
\mathbb{E}[d-\hat{d}] &=   (V_1 M_1 C_1+ V_2 M_2 (C_2 \hat{A}+T_2 \dot{C}))  \mathbb{E}[\tilde{x}]   \\
&\quad +  \frac{1}{2}V_2 M_2 \mathbb{E}[\ddot{z}_2(c)] \mathfrak{d}t,\\
\Rightarrow \| \mathbb{E}[\tilde{d}] \|&\leq  \| (V_1 M_1 C_1+ V_2 M_2 (C_2 \hat{A}+T_2 \dot{C})) \| \beta e^{-\gamma (t-t_0)}   \\
&\quad +  \frac{1}{2} \|V_2 M_2 \mathbb{E}[\ddot{z}_2(c)]\| \mathfrak{d}t \\
& \leq  \alpha_1 e^{-\gamma (t-t_0)} +  \alpha_2 \mathfrak{d}t,\vspace{-0.45cm}\end{array} \label{eq:E_d}
\end{align} \noindent
where $\alpha_1$ and $\alpha_2$ are given in \eqref{eq:alpha1} and \eqref{eq:alpha23}. 

Next, we find the approximation error induced by the finite difference approximation on the input error covariance matrix. Furthermore, we have ${\rm tr}(\mathbb{E}[\tilde{d} \tilde{d}^\top])={\rm tr} (P^d_1)+ {\rm tr}(P^d_2)$ (shown later in \eqref{eq:trace}). Since there is no approximation error in the estimate of $\hat{d}_1$ because it is independent of $\dot{z}_2$, the only source of approximation error comes from the error covariance matrix $P^{d}_2$, which can be computed from \small
\begin{align}
\begin{array}{ll}
 P^d_2 \delta(0)=\mathbb{E}[(d_2- \hat{d}_2)(t) (d_2- \hat{d}_2)^\top (t)]\\
 =  \mathbb{E}[(d_2- \overline{d}_2)(t)(\overline{d}_2- \hat{d}_2)^\top(t)]+\mathbb{E}[(\overline{d}_2- \hat{d}_2)(t) (d_2- \overline{d}_2)^\top(t)] \\ 
 \ +\mathbb{E}[(\overline{d}_2- \hat{d}_2)(t)(\overline{d}_2- \hat{d}_2)^\top(t)]  +\mathbb{E}[(d_2- \overline{d}_2)(t)(d_2- \overline{d}_2)^\top(t)] \\ 
 =\frac{1}{2}\mathbb{E}[\tilde{d}_2 \ddot{z}_2(c)^\top]  M_2^\top \mathfrak{d}t + \frac{1}{2} M_2 \mathbb{E}[\ddot{z}_2(c) \tilde{d}_2^{\top}] \mathfrak{d}t \\
 \quad +  \frac{1}{4} M_2 \mathbb{E}[\ddot{z}_2(c)\ddot{z}_2(c)^\top] M_2^\top (\mathfrak{d}t)^2+M_2 \tilde{R}_2 M_2^\top  (t) \delta(0) \\ 
=\frac{1}{2}\mathbb{E}[\tilde{d}_2 \ddot{z}_2(c)^\top]  M_2^\top \mathfrak{d}t + \frac{1}{2} M_2 \mathbb{E}[\ddot{z}_2(c) \tilde{d}_2^{\top}] \mathfrak{d}t \\
\quad +  \frac{1}{4} M_2 \mathbb{E}[\ddot{z}_2(c)]\mathbb{E}[\ddot{z}_2(c)]^\top M_2^\top (\mathfrak{d}t)^2\\
 \quad +  \frac{1}{4} M_2 \mathbb{E}[(\ddot{z}_2(c)-\ddot{z}_2(c)])(\ddot{z}_2(c)-\mathbb{E}[\ddot{z}_2(c)]^\top) M_2^\top (\mathfrak{d}t)^2\\
 \quad +M_2 \tilde{R}_2 M_2^\top  (t) \delta(0) \\ 
= ( (\frac{1}{4} M_2 T_2(\begin{bmatrix} -A_v & -A_{\dot{v}} \end{bmatrix}P^v \begin{bmatrix} -A_v & -A_{\dot{v}} \end{bmatrix}^\top\hspace{-0.1cm}+\hspace{-0.05cm} C W B_w Q_G B_w^\top W^\top C^\top\\
\quad +B_v R_G B_v^\top +(\dot{C}W + C \dot{W} - C WA_w)P^w(\dot{C}W + C \dot{W} \\
\quad - C WA_w)^\top) T_2^\top M_2^\top) (c) (\mathfrak{d}t)^2  +M_2 \tilde{R}_2 M_2^\top (t)) \delta(0), 
 \end{array} \label{eq:Pd2ALISE}
\end{align} \noindent \normalsize
where we applied Lemma \ref{lem:Taylor}, \eqref{eq:dtilde}, and removed the negligible 
contributions of $\mathbb{E}[\tilde{d}_2 \ddot{z}_2(c)^\top]$ and $\mathbb{E}[\ddot{z}_2(c)]\mathbb{E}[\ddot{z}_2(c)^\top]$ since $c \neq t $ such that $\delta(t-c) \ll \delta(0)$ , and both $\mathbb{E}[\tilde{d}_2]$ and $\mathbb{E}[\ddot{z}_2]$ are finite. Thus, comparing the above error covariance matrix \eqref{eq:Pd2ALISE} with the input estimate error covariance matrix with perfect knowledge of $\dot{z}_2$ given by $P^{\overline{d}}_2:=M_2 \tilde{R}_2 M_2^\top$, 
we can find the trace of the difference between the two error covariance matrices:
\begin{align}
\begin{array}{ll}
|{\rm tr} (P^d - P^{\overline{d}})| = |
 \frac{1}{4} {\rm tr}(M_2 T_2(C W B_w Q_G B_w^\top  W^\top C^\top\\
\qquad  +[-A_v \ -A_{\dot{v}}] P^v [-A_v \ -A_{\dot{v}}]^\top + B_v R_G B_v^\top  \\
\qquad +(\dot{C}W + C \dot{W} - C WA_w)P^w(\dot{C}W + C \dot{W} \\
\qquad - C WA_w)^\top    )   T_2^\top M_2^\top)(c)|  (\mathfrak{d}t)^2   \leq \alpha_3  (\mathfrak{d}t)^2, 
\end{array}\label{eq:tracePd}
 \end{align} \noindent
where $\alpha_3$ is as given in \eqref{eq:alpha23}.
Since $\alpha_1$, $\alpha_2$ and $\alpha_3$ in \eqref{eq:E_d} and \eqref{eq:tracePd} are bounded by Assumption ($A1'$) and by choice of the noise models which results in finite noise intensities (see bounds in \cite{Hmamed.1990}), 
and $\mathfrak{d}t$ can be chosen to be arbitrarily small, the expected value of the estimate $\hat{d}_2$ given by \eqref{eq:variant2} asymptotically tends to the true value of $d_2$ with minimum-variance error covariance to any desired accuracy. Thus, Lemma \ref{lem:convergenceInput} holds.
\subsection{Proof of Theorem \ref{thm:main}}%

We first provide a proof of Theorem \ref{thm:main} by constructing a `virtual' equivalent system without unknown inputs, which allows us to derive analogous properties of
our filter to that of the Kalman-Bucy filter. Then, we provide an alternative derivation for the Special Case \ref{sp} by means of limiting case approximations of the optimal discrete-time filter \cite{Yong.Zhu.ea.Automatica16}. In the process, we gain insight into the subtle difference between the special case continuous-time filter and the discrete-time filter in \cite{Yong.Zhu.ea.Automatica16}. In particular, we observed that the optimal discrete-time filter implicitly uses finite difference to obtain an `output derivative'. 
In the case with a biased initial state estimate, the associated state and unknown input bias decays exponentially as shown in Lemmas \ref{lem:convergence} and \ref{lem:convergenceInput}.

\subsubsection{Proof 1: By Equivalent System without Unknown Inputs} \label{sec:proof1}

In this first proof of Theorem \ref{thm:main}, we construct 
a `virtual' equivalent system without unknown inputs, for which analogous results of the Kalman-Bucy filter  \cite{kalman.bucy.1961} can be inferred. To this end, as was also observed in \cite{Yong.Zhu.ea.Automatica16}, we view the unknown input as consisting of a known component given by the input estimate, and a zero-mean random variable with known variance which can be dealt in the same manner as with process and measurement
noise signals:
\begin{align} \label{eq:d}
\begin{array}{l}
d_1=\hat{d}_1+(d_1-\hat{d}_1):=\hat{d}_1+ \tilde{d}_1,\\
d_2=\hat{d}_2+(d_2-\hat{d}_2):=\hat{d}_2+ \tilde{d}_2.
\end{array}
\end{align} \noindent
Since $\mathbb{E}[\tilde{x}]$ decays exponentially (Lemma \ref{lem:convergence}), and the process and measurement noises have zero mean, the expected values of both $\tilde{d}_1$ and $\tilde{d}_2$ exponentially tend towards zero-mean random variables with the following (auto-)correlations:
\begin{align}
&\begin{array}{ll}
&\mathbb{E}[\tilde{d}_1(t) \tilde{d}_1(t')^\top] :=P^d_1 (t) \delta(t-t')\\&\quad =(M_1 R_1 M_1^\top+M_1 C_1 P^x C_1^\top M_1^\top)(t) \delta(t-t'), \end{array} \label{eq:Pd1} \\
&\begin{array}{ll}& \mathbb{E}[\tilde{d}_1(t) \tilde{d}_2(t')^\top]:=P^d_{12} (t) \delta(t-t')\\& \quad=(M_1 C_1 P^x (\hat{A}^\top \overline{C}_2^\top+ \doverline{C}^\top \overline{T}_2^\top) M_2^\top + M_1 \grave{R}_{12} M_2^\top\\
& \qquad - M_1 R_1 M_1^\top G_1^\top \overline{C}_2^\top M_2^\top)(t)\delta(t-t'),\end{array} \label{eq:Pd12}\\
& \begin{array}{ll}&\mathbb{E}[\tilde{d}_2(t) \tilde{d}_2(t')^\top]:=P^d_2 (t) \delta(t-t')\\ &\quad =(M_2 \tilde{R}_2 M_2^\top)(t) \delta(t-t'),\end{array} \label{eq:Pd2} \\
&\begin{array}{ll} &\mathbb{E}[\tilde{d}(t) \tilde{d}(t')^\top] \hspace{0.3cm}:=P^d (t) \delta(t-t')\\
 &\quad =(V_1 P^d_1 V_1^\top +V_1 P^d_{12} V_2^\top+V_2 P^{d\top}_{12} V_1^\top\\
 &\qquad +V_2 P^d_2 V_2^\top)(t) \delta(t-t'),\end{array} \label{eq:Pd}
\end{align} \noindent
where we defined $\mathbb{E}[\tilde{x}(t) \tilde{x}^\top (t')]:=P^x (t) \delta(t-t')$, $\tilde{R}_2:=(\overline{C}_2 \hat{A}+\overline{T}_2 \doverline{C}) P^x (\overline{C}_2 \hat{A}+\overline{T}_2 \doverline{C})^\top + \overline{C}_2 \hat{Q} \overline{C}_2^\top+\overline{R}_2-\grave{R}_{12}^\top M_1^\top G_1^\top \overline{C}_2^\top-\overline{C}_2 G_1 M_1 \grave{R}_{12}$ and $\hat{Q}:=W Q W^\top+G_1 M_1 R_1 M_1^\top G_1^\top$ as well as omitted $\mathbb{E}[\tilde{x} (t) v_1^\top (t')]$, $\mathbb{E}[\tilde{x}(t) \overline{v}_2^\top(t')]$ and $\mathbb{E}[\tilde{x}(t) w^\top(t')]$ due to their negligible contributions to the above correlations.

To obtain the best linear unbiased estimate of both projections of the unknown inputs, $\hat{d}_1$ and $\hat{d}_2$, as in its discrete-time counterpart \cite{Yong.Zhu.ea.Automatica16}, we choose $M_1$ and $M_2$ such that the assumption in 
the Gauss-Markov Theorem is satisfied, as outlined in \cite[pp. 96-98]{kailath.2000}:
\begin{align}
M_1 =\Sigma^{-1}, \ 
M_2 = (G_2^\top C_2^\top \tilde{R}_2^{-1} C_2 G_2)^{-1} G_2^\top C_2^\top \tilde{R}_2^{-1}.\label{eq:M}
 \end{align} \noindent

Next, we note the following equality: 
\begin{align} \label{eq:trace}
\begin{array}{l}
{\rm tr} (\mathbb{E}[\tilde{d} \tilde{d}^\top])={\rm tr} (\mathbb{E}[V \begin{bmatrix} \tilde{d}_{1} \\ \tilde{d}_{2} \end{bmatrix} \begin{bmatrix} \tilde{d}_{1}^{ \, \top} & \tilde{d}_{2}^{ \, \top} \end{bmatrix} V^\top])\\
= {\rm tr} (V^\top V \mathbb{E}[\begin{bmatrix} \tilde{d}_{1} \\ \tilde{d}_{2} \end{bmatrix} \begin{bmatrix} \tilde{d}_{1}^{ \, \top} & \tilde{d}_{2}^{ \, \top} \end{bmatrix}])={\rm tr} (P^d_{1})+{\rm tr} (P^d_{2}).\end{array}
\end{align} \noindent
Since the unbiased estimate of $\hat{d}_{1}$ is unique, the minimum of \eqref{eq:trace} is given by  $\min {\rm tr} (\mathbb{E}[\tilde{d} \tilde{d}^\top])= {\rm tr} (\mathbb{E}[\tilde{d}_{1} \tilde{d}_{1}^\top]) + \min {\rm tr} (\mathbb{E}[\tilde{d}_{2} \tilde{d}_{2}^\top])$, from which it can be observed that the unbiased estimate $\hat{d}$ has minimum variance when $\hat{d}_{1}$ and $\hat{d}_{2}$ have minimum variances.

Note that even during transients, where the $\tilde{d}_1$ and $\tilde{d}_2$ have non-zero means, the terms contributing to these biases are functions of $\tilde{x}$ and are thus absorbed into the $\overline{A}$ as seen in \eqref{eq:xtilde}. More importantly, the state estimate error dynamics in \eqref{eq:xtilde} is the same as that of 
a Kalman-Bucy filter  \cite{kalman.bucy.1961} for a `virtual' linear system without unknown inputs given by 
\begin{align} \label{eq:equiv1}
\begin{array}{ll}
\dot{x}_e&=\overline{A} x_e + \overline{w},\\
y_e&= C_2 x_e + v_2,
\end{array}
\end{align}
where $\overline{A}$ and $\overline{w}$ are as defined in Theorem \ref{thm:stable} and the noise terms are correlated, i.e., $\mathbb{E}[\overline{w}(t) v_2^\top(t')]=-G_2 M_2 \grave{R}_2^\top (t) \delta(t-t')$.
Since the objectives of both systems are the same, i.e. to obtain an unbiased minimum-variance filter, they are equivalent systems from the perspective of optimal filtering. Hence, the optimal filter is as with the Kalman-Bucy filter with correlated noise (see, e.g., \cite{kailath.2000,Crassidis.2004}), 
i.e., with 
\begin{align} \label{eq:L}
L=(P^x C_2^\top-G_2 M_2 \grave{R}_2^\top) R_2^{-1},
\end{align}
and the state estimate error covariance, $P^x$, is obtained from the Riccati differential equation:
\begin{align} \label{eq:riccati}
\dot{P}^x=\overline{A} P^x + P^x \overline{A}^\top +\overline{Q} - L R_2 L^\top,
\end{align} \noindent
where the noise intensity, $\overline{Q}$, is given in Theorem \ref{thm:stable}. 

In summary, the proposed filter provides the best linear unbiased estimate of the unknown input and the minimum-variance unbiased estimate of the state; thus, Theorem \ref{thm:main} holds.

\subsubsection{Proof 2: By Limiting Case Approximations}
An alternate derivation of the optimal filter can be obtained for Special Case \ref{sp} from the optimal discrete-time filter \cite{Yong.Zhu.ea.Automatica16} using limiting case approximations. Although this derivation lacks rigor due to various approximations, this is interesting from a pedagogical point of view, since this is often used to derive the continuous-time Kalman-Bucy filters in textbooks (e.g. \cite{Simon.2006}). 

If the sampling period $\Delta t$ is small, we can use Euler's approximation to write the discretized version of \eqref{eq:sys} as
\begin{align} \label{eq:approx}
&\begin{array}{ll}
x_{k} &\approx(I + A \Delta t) x_{k-1} + B \Delta t u_{k-1} + G_1 \Delta t d_{1,k-1} \\
&\qquad+ G_{2} \Delta t d_{2,k-1} + W w_{k-1}\\
&:=A_{k-1} x_{k-1} + B_{k-1} u_{k-1} + G_{1,k-1} d_{1,k-1} \\
&\qquad+ G_{2,k-1} d_{2,k-1} + W_{k-1} w_{k-1}, \\ 
y_k&= C x_k + D u_k + H_{1}  d_k + v_k\\
&:= C_k x_k + D_k u_k + H_{1,k}  d_k + v_k,\\
 z_{1,k}& = C_{1} x_k + D_{1} u_k +\Sigma d_{1,k} + v_{1,k} \\
 &:= C_{1,k} x_k + D_{1,k} u_k +\Sigma_k d_{1,k} + v_{1,k},\\
z_{2,k}& = C_{2} x_k + D_{2} u_k + v_{2,k} \\
&:= C_{2,k} x_k + D_{2,k} u_k + v_{2,k},
\end{array}
\end{align} \noindent
where the process and measurement noises are $w_k \sim (0, Q \Delta t)$ and $v_k \sim (0, R / \Delta t)$, in which $Q_k \approx Q \Delta t$ as $\Delta t \to 0$, and the discrete measurement noise is approximated as the average value of the continuous noise \cite{Crassidis.2004}.
\balance

Since the first component of the unknown input can be computed pointwise without delay, we expect $M_{1,k} \to M_1$. Thus, we have the estimate $\hat{d}_1$ as in \eqref{eq:dhat} directly from the discrete-time version given by $\hat{d}^D_{1,k}=M_{1,k} (z_{1,k}-C_{1,k} \hat{x}_{k|k}-D_{1,k} u_k)$ \cite{Yong.Zhu.ea.Automatica16}. On the other hand, the limiting case approximation of the second component of the unknown input is given by:
\begin{align*}
\begin{array}{ll}
\hat{d}^D_{2,k-1}&=M_{2,k}(z_{2,k}-C_{2,k} (A_{k-1}  \hat{x}_{k-1|k-1} + B_{k-1} u_{k-1} \\
&\qquad + G_{1,k-1} \hat{d}_{1,k-1}) - D_{2,k} u_{k})\\
&\approx M_{2,k} \Delta t ( T_{2,k}\frac{y_{k}-C_{k}\hat{x}_{k-1|k-1}-D_{k} u_{k} }{\Delta t}   \\
&\qquad-C_{2,k} (A \hat{x}_{k-1|k-1}+B u_{k-1} + G_{1} \hat{d}_{1,k-1})) \\ 
& = M_{2,k} \Delta t \left[ T_{2,k}( \frac{y_{k}-\hat{y}_{k-1}}{\Delta t} +\frac{(C_{k-1} -C_k) \hat{x}_{k-1|k-1}}{\Delta t} \right. \\
&\qquad+ \frac{D_{k-1} u_{k-1}-D_{k} u_k}{\Delta t} + \frac{H_{1,k-1} d_{1,k-1}}{\Delta t}) \\
& \qquad \left. -C_2 A \hat{x}_{k-1|k-1} -C_{2} B u_{k-1}- C_{2} G_{1} \hat{d}_{1,k-1}  ) \right] \\ 
& = M_{2,k} \Delta t \left[\frac{z_{2,k}-\hat{z}_{2,k-1}}{\Delta t} + T_{2,k} \frac{(C_{k-1} -C_k)} {\Delta t} \hat{x}_{k-1|k-1}\right. \\
&\qquad- \frac{(D_k-D_{k-1}) u_{k-1}+D_{k} (u_k-u_{k-1})}{\Delta t} \\
&\qquad + T_{2,k} \frac{(H_{1,k} +(H_{1,k-1}-H_{1,k})) d_{1,k-1}}{\Delta t}   -C_2 A \hat{x}_{k-1|k-1} \\
&\qquad \left.-C_{2} B u_{k-1}- C_{2} G_{1} \hat{d}_{1,k-1}  ) \right],
\end{array}
\end{align*}
where the first equation is the discrete-time version from \cite{Yong.Zhu.ea.Automatica16}, and we substituted the approximate matrices $A_{k-1} \approx I + A \Delta t$, $B_{k-1} \approx B \Delta t$, $C_{2,k} \approx C_2$, $D_{2,k} \approx D_2 \Delta t$ and $G_{1,k-1} \approx G_1 \Delta t$ from \eqref{eq:approx}. We also defined $M_2 :=\lim_{\Delta t \to 0} M_{2,k} \Delta t$, $\hat{y}_{k-1}:=C_{k-1} \hat{x}_{k-1|k-1}+D_{k-1} u_{k-1}$ and $\hat{z}_{2,k-1}:=T_{2,k} \hat{y}_{k-1}$. We can further simplify the above equation by noticing that $T_{2,k} H_{1,k}=0$. Taking the limit of $\Delta t \to 0$, we obtain
\begin{align} \label{eq:dhat_c}
\begin{array}{rl}
 \hat{d}_{2}&= M_{2}(\overline{z}_2-C_{2} A \hat{x} -C_{2} B u - C_{2} G_{1} \hat{d}_{1} -T_2 \dot{C} \hat{x}\\ & \qquad -T_2 \dot{D} u-D_{2}\dot{u}),\end{array}
\end{align} \noindent
where we replaced the term $\lim_{\Delta t \to 0} \frac{z_{2,k}-\hat{z}_{2,k-1}}{\Delta t}= T_2 (\lim_{\Delta t \to 0} \frac{y_{k}-{y}_{k-1}}{\Delta t}+\frac{C_{k-1} \tilde{x}_{k-1|k-1}+v_{k-1}}{\Delta t})$ with $\overline{z}_2=T_2 \overline{y}=T_2 (\lim_{\Delta t \to 0}  \frac{y_{k}-y_{k-1}}{\Delta t} +\frac{\overline{v}_{k} \Delta t - (v_k -v_{k-1})}{\Delta t})$, 
which we assume is obtained from the noisy measurement of $\overline{y}$ according to \eqref{eq:sys}, applied Corollary \ref{cor:cor1} 
and defined $\tilde{x}_{k-1|k-1}:=x_{k-1}-\hat{x}_{k-1|k-1}$. This indirectly implies that the optimal discrete-time filter ``differentiates" the second projection of the output, $z_2$, using finite difference. Moreover, we can infer that the equivalent discrete-time estimation of ${d}_{2,k-1}$ corresponding to \eqref{eq:dhat_c} is 
\begin{align} \label{eq:dhat2new}
\begin{array}{ll} \hat{d}_{2,k-1}&=M_{2,k}(z_{2,k}-C_{2,k} (A_{k-1}  \hat{x}_{k-1|k-1}  \\& \qquad + B_{k-1} u_{k-1}+ G_{1,k-1} \hat{d}_{1,k-1}) - D_{2,k} u_{k}\\
&\qquad -T_{2,k} C_{k-1} \tilde{x}_{k-1|k-1}+T_{2,k}\overline{v}_{k} \Delta t- v_{2,k}),
\end{array}
\end{align} \noindent
which would be not implementable because the noise terms and the true state are not available. Next, to obtain the best linear unbiased estimate of both projections of the unknown inputs, $\hat{d}_{1,k}=\hat{d}^D_{1,k}$ and $\hat{d}_{2,k-1}$, we choose $M_{1,k}$ and $M_{2,k}$ as in \cite{Yong.Zhu.ea.Automatica16}, such that $M_{1,k}\Sigma_k =I$, $M_{2,k} C_{2,k} G_{2,k-1}=I$ and the Gauss-Markov Theorem is satisfied \cite[pp. 96-98]{kailath.2000}:
\begin{align*} 
\begin{array}{rl}
M_{1,k} &=\Sigma^{-1}_k,\\
 M_{2,k} &= (G_{2,k-1}^\top C_{2,k}^\top \tilde{R}_{2,k}^{-1} C_{2,k} G_{2,k-1})^{-1} G_{2,k-1}^\top C_{2,k}^\top \tilde{R}_{2,k}^{-1},
\end{array} 
 \end{align*} \noindent
 where $ \tilde{R}_{2,k}:=C_{2,k} \hat{Q}_{k-1} C_{2,k}^\top +[T_{2,k} (C_{k}-C_{k-1})+C_{2,k} (\hat{A}_{k-1} -I)] P^x_{k-1|k-1} [ (C_{k}-C_{k-1})^\top T_{2,k}^\top+(\hat{A}_{k-1} -I)^\top C_{2,k}^\top]+  T_{2,k} \overline{R} T_{2,k}^\top \Delta t - C_{2,k} G_{1,k-1} M_{1,k-1} T_{1,k-1} \grave{R} T_{2,k}^\top -T_{2,k} \grave{R}^\top T_{1,k-1}^\top M_{1,k-1}^\top G_{1,k-1}^\top C_{2,k}^\top$, $\hat{A}_{k}:=A_{k} - G_{1,k} M_{1,k} C_{1,k}$ and $\hat{Q}_{k}:=W_k Q_{k} W_k^\top +G_{1,k} M_{1,k} R_{1,k} M_{1,k}^\top G_{1,k}^\top$; and we applied $\mathbb{E}[\overline{v}_{k} \overline{v}_{k}^\top] := \overline{R}_{k} \approx \frac{\overline{R}}{\Delta t}$ and $\mathbb{E}[{v}_{k-1} \overline{v}_{k}^\top] := \grave{R}_{k-1} \approx \frac{\grave{R}}{\Delta t}$.
Then, substituting the approximate matrices as before, 
as well as defining 
$\dot{C} := \lim_{\Delta t \to 0} \frac{C_{k}-C_{k-1}}{\Delta t}$ and approximating $P^x_{k-1|k-1} \approx \frac{P^x}{\Delta t}$ and $\tilde{R}_{2,k} \approx \frac{\tilde{R}_2}{\Delta t}$, followed by taking the limit of $\Delta t \to 0$, we obtain the filter gains $M_1= \lim_{\Delta t \to 0} M_{1,k}$ and $M_2=\lim_{\Delta t \to 0} M_{2,k} \Delta t$ given by \eqref{eq:M}. 

In addition, by applying the approximations defined in \eqref{eq:approx} to the dynamics of the discrete-time filter proposed in \cite{Yong.Zhu.ea.Automatica16} (ULISE), neglecting higher order terms and taking $\Delta t \to 0$, we obtain the the state estimate dynamics given in \eqref{eq:xhat}, the filter gain $L$ given in \eqref{eq:L} and the Riccati differential equation governing $P^x$
given in \eqref{eq:riccati}, as well as $P^d_1:= \lim_{\Delta t \to 0} P^d_{1,k}\Delta t$, $P^d_{12}:= \lim_{\Delta t \to 0} P^d_{12,k-1}\Delta t$ and $P^d_2:= \lim_{\Delta t \to 0} P^d_{2,k-1}\Delta t$ given by \eqref{eq:Pd1}, \eqref{eq:Pd12} and \eqref{eq:Pd2}. The detailed derivations of these equations follow the same approach as in \cite[Section C]{Yong.Zhu.Frazzoli.ACC15} and are omitted due to space limitations. From the perspective of limiting case approximations, the discrete-time filter in \cite{Yong.Zhu.ea.Automatica16} is globally optimal and converges to a steady-state solution for arbitrary $\Delta t$, and the Euler approximation converges to the continuous system. So, from the optimality of the Kalman-Bucy filter, it can be inferred that the limiting case filter is also optimal.

\subsection{Proof of Theorem \ref{thm:main2}}%

Since we have shown the equivalence of the state estimation of ELISE and ALISE in Appendix \ref{sec:equivAlg}, the optimality of the state estimation of ALISE in the minimum-variance unbiased sense given in Theorem \ref{thm:main2} follows directly from Theorem \ref{thm:main}. If the initial state estimate is biased, Lemmas \ref{lem:convergence}, \ref{lem:convergenceInput}, \ref{lem:convergence2} and \ref{lem:convergenceInput2} show that the associated state and unknown input biases decay exponentially.

\subsection{Proof of Theorems \ref{thm:stable} and \ref{thm:stable2}} \label{sec:proofThm4}
In Appendix \ref{sec:proof1}, we showed that the system with unknown inputs in \eqref{eq:system} is equivalent to the `virtual' system without unknown input in \eqref{eq:equiv1}. However, the noise terms of this new form are correlated, i.e., $\mathbb{E}[\overline{w}(t) v_2\top(t')]=-G_2 M_2  \grave{R}_2^\top \neq 0$. Hence, we further transform the system into one with uncorrelated noise terms by employing a common trick (cf., e.g., \cite[p. 182]{Crassidis.2004}) of adding a zero term ($y_e-C_2 x_e-v_2=0$) to obtain yet another `virtual' equivalent system in \eqref{eq:eqSys}.
 Thus, we can analogously apply the results of the Kalman-Bucy filter \cite[Theorem 4]{kalman.bucy.1961} to obtain the necessary assumption (A2-A5) given in Theorems \ref{thm:stable} and \ref{thm:stable2}, such that the optimal filter is uniformly asymptotically stable, and that the variance equation converges to a unique behavior for large $t$, independent of $\mathcal{P}^x_0$.

\subsection{Proof of Theorems \ref{thm:conv} and \ref{thm:conv2}}

For linear time-invariant systems, the conditions for the convergence of the filter gains to steady-state of the proposed filter are closely related to the existence and uniqueness of stabilizing solutions of its continuous-time algebraic Riccati equation (CARE), i.e. \eqref{eq:riccati} with $\dot{P^x}=0$. Since we have shown in Appendix \ref{sec:proof1} and \ref{sec:proofThm4} that we can transform the estimation of a system with unknown inputs to a `virtual' equivalent system with no unknown inputs, analogous convergence properties to the steady-state Kalman-Bucy filter apply, as summarized in Theorem \ref{thm:conv}. For a proof of the results of Kalman-Bucy convergence properties, the reader is referred to \cite{kailath.2000}.

\subsection{Proof of Proposition \ref{remark}}

The connection between strong observability and the observability of $(\overline{A},C_2)$, as well as $C_2$ and $G_2$ being full rank 
follow directly from 
%
{\fontsize{9}{9}\selectfont
\begin{align*}
&n+p={\rm rk}\begin{bmatrix} sI-A & -G \\ C & H \end{bmatrix} ={\rm rk}\begin{bmatrix} sI-A & -G \\ C & U \begin{bmatrix} \Sigma & 0 \\ 0 & 0 \end{bmatrix} V^\top \end{bmatrix} \\ &= {\rm rk} \begin{bmatrix} I & 0 \\ 0 & T \end{bmatrix} \begin{bmatrix} sI-A & -G \\ C & U \begin{bmatrix} \Sigma & 0 \\ 0 & 0 \end{bmatrix} V^\top \end{bmatrix} \begin{bmatrix} I & 0 \\ 0 & V \end{bmatrix} 
\hspace{-0.1cm}=\hspace{-0.05cm} {\rm rk}\hspace{-0.05cm}  \begin{bmatrix} sI-A & -G_1 & -G_2\\ C_1 & \Sigma & 0\\ C_2 & 0 & 0 \end{bmatrix} \hspace{-0.1cm}\\
&=\hspace{-0.05cm}{\rm rk}\hspace{-0.05cm} \begin{bmatrix} I & G_1\Sigma^{-1} & 0 \\ 0 & I & 0\\ 0 & 0 & I\end{bmatrix}\hspace{-0.05cm} \begin{bmatrix} sI-A & -G_1 & -G_2 \\ C_1 & \Sigma &0 \\ C_2 & 0 & 0 \end{bmatrix} \hspace{-0.1cm}=\hspace{-0.05cm}{\rm rk}\hspace{-0.05cm}\begin{bmatrix} sI-\hat{A} & 0 & -G_2 \\ C_1 & \Sigma & 0\\ C_2 & 0 & 0 \end{bmatrix} \end{align*} \begin{align*} 
&={\rm rk}\begin{bmatrix} sI-\hat{A} & -G_2 \\ C_2 & 0 \end{bmatrix}\hspace{-0.05cm}+\hspace{-0.05cm}p_H \hspace{-0.05cm}=\hspace{-0.05cm}{\rm rk}\begin{bmatrix} sI-\hat{A} & -G_2 \\ C_2 & 0 \end{bmatrix}\hspace{-0.05cm}\begin{bmatrix} I  & 0 \\ -M_2 C_2 \hat{A} & I \end{bmatrix}\hspace{-0.05cm}+\hspace{-0.05cm}p_H\\& 
={\rm rk}\begin{bmatrix} sI-\overline{A} & -G_2 \\ C_2 & 0 \end{bmatrix}\hspace{-0.05cm}+\hspace{-0.05cm}p_H,
\end{align*}}
\noindent \hspace{-0.2cm}where the first equality is the rank condition for strong observability given in \cite{Hautus.1983}, the third from last equality holds because $\Sigma$ is square and has full rank $p_H$, and we have assumed that $n \geq l \geq 1$ and $l \geq p \geq 0$.

\end{document}